\documentclass[3p,final,10pt]{elsarticle}
\usepackage[utf8]{inputenc}
\usepackage{amsmath,amsthm,amsfonts,mathtools,mathrsfs,enumitem}
\usepackage{tikz}
\usetikzlibrary{shapes,arrows,backgrounds,positioning}
\usetikzlibrary{calc,arrows,decorations.markings,decorations.pathreplacing}
\usepackage{cleveref}
\crefname{equation}{}{}
\makeatletter
\newenvironment{proofof}[1]{\par
  \pushQED{\qed}%
  \normalfont \topsep6\p@\@plus6\p@\relax
  \trivlist
  \item[\hskip\labelsep
        \bfseries
    Proof of #1\@addpunct{.}]\ignorespaces
}{%
  \popQED\endtrivlist\@endpefalse
}
\makeatother

\newcommand{\ve}{\varepsilon}
\newcommand{\R}{\mathbb{R}}
\newcommand{\N}{\mathbb{N}}
\newcommand{\Cinf}{\mathcal{C}^{\infty}}
\newcommand{\Cl}{\mathcal{C}^{\ell}}

\newcommand{\M}{\mathcal{M}}
\newcommand{\Sen}{\Sigma^{\en}}

\newcommand{\Wc}{\mathcal{W}^{^C}}

\newcommand{\eq}[1]{ 
\begin{equation}
  \begin{split}
    #1 
  \end{split}
\end{equation}
}

\newcommand{\parc}[1]{\dfrac{\partial}{\partial #1}} 
\newcommand{\parcs}[2]{\dfrac{\partial #1}{\partial #2}}

\newcommand{\tparc}[1]{\tfrac{\partial}{\partial #1}} 
\newcommand{\tparcs}[2]{\tfrac{\partial #1}{\partial #2}}  

\newtheorem{remark}{Remark}[section] 
 
\newtheorem{definition}{Definition}[section] 
\newtheorem{theorem}{Theorem}[section] 
\newtheorem{lemma}{Lemma}[section] 
\newtheorem{proposition}{Proposition}[section] 
\newtheorem{corollary}{Corollary}[section]

\renewcommand{\H}{\mathcal{H}}

\renewcommand{\r}{\rangle}

\DeclareMathOperator{\en}{en}
\DeclareMathOperator{\ex}{ex}

\DeclareMathOperator{\inn}{inner}
\DeclareMathOperator{\Graph}{Graph}
\DeclareMathOperator{\Supp}{Supp}
\DeclareMathOperator{\Div}{div}
\DeclareMathOperator{\Id}{Id}
\begin{document}

\begin{frontmatter}
  \title{Analysis of a slow-fast system near a cusp singularity}
  \author[jbi,cor1]{H. Jardón-Kojakhmetov}
  \ead{h.jardon.kojakhmetov@rug.nl}
  \author[jbi]{Henk. W. Broer}
  \ead{h.w.broer@rug.nl}
  \author[bur]{R. Roussarie}
  \ead{Robert.Roussarie@u-bourgogne.fr}
  \address[jbi]{Johann Bernoulli Institute for Mathematics and Computer Science, University of Groningen, P.O. Box 407, 9700 AK, Groningen, The Netherlands.}
  \address[bur]{Institut de Mathématique de Bourgogne, U.M.R. 5584 du C.N.R.S., Université de Bourgogne, B.P. 47 870, 21078 Dijon Cedex, France. }
  \cortext[cor1]{Corresponding author}
  \begin{abstract}
    This paper studies a slow-fast system whose principal characteristic is that the slow manifold is given by the critical set of the cusp catastrophe. Our analysis consists of two main parts: first, we recall a formal normal form suitable for systems as the one studied here; afterwards, taking advantage of this normal form, we investigate the transition near the cusp singularity by means of the blow up technique. Our contribution relies heavily in the usage of normal form theory, allowing us to refine previous results.
  \end{abstract}

\end{frontmatter}

\tableofcontents

\section{Introduction}\label{sec:intro}

A \emph{slow-fast system} (SFS) is a singularly perturbed ordinary differential equation of the form
\eq{\label{intro:sf1}
  \dot x &= f(x,z,\ve)\\
  \ve\dot z &= g(x,z,\ve),
}

where $x\in\R^m$, $z\in\R^n$ are local coordinates and where $\ve>0$ is a small parameter. The over-dot denotes the derivative with respect to the time parameter $t$. Throughout this text, we assume that the functions $f$ and $g$ are of class $\Cinf$. In applications (e.g \cite{Zeeman1}), $z(t)$ represents states or measurable quantities of a process while $x(t)$ stands for control parameters. The parameter $\ve$ models the difference of the rates of change between the variables $z$ and $x$. That is why systems like \cref{intro:sf1} are often used to model phenomena with two time scales. Observe that the smaller $\ve$ is, the faster $z$ evolves with respect to $x$. Therefore we refer to $x$ (resp. $z$) as the \emph{slow} (resp. \emph{fast}) variable. The time parameter $t$ is known as the \emph{slow time}. For $\ve\neq 0$, we can define a new time parameter $\tau$ by the relation $t=\ve\tau$. With this time reparametrization \cref{intro:sf1} can be written as
\eq{\label{intro:sf2}
x' &= \ve f(x, z,\ve)\\
z' &= g(x,z,\ve),
}

where now the prime denotes the derivative with respect to the rescaled time parameter $\tau$, which we call \emph{the fast time}. Since we consider only autonomous systems, we often omit to indicate the time dependence of the variables. In the rest of this document, we prefer to work with slow-fast systems presented as \cref{intro:sf2}.\smallskip

Observe that as long as $\ve\neq 0$ and $f$ is not identically zero, systems \cref{intro:sf1} and \cref{intro:sf2} are equivalent. A first approach to understand the qualitative behavior of slow-fast systems is to study the limit $\ve\to0$. The slow equation \cref{intro:sf1} restricted to $\ve=0$ reads as
\eq{\label{intro:cde1}
\dot x &= f(x,z,0)\\
0 &= g(x,z,0).
}

A system of the form \cref{intro:cde1} is called \emph{constrained differential equation} (CDE) \cite{Jardon1,Takens1}. On the other hand, in the limit $\ve\to0$, a system given by \cref{intro:sf2} becomes
\eq{\label{intro:layer}
x' &= 0\\
z' &= g(x,z,0),
}

which is called \emph{the layer equation}. Associated to both systems, \cref{intro:cde1} and \cref{intro:layer}, the slow manifold $S$ is defined by
\eq{
S=\left\{ (x,z)\in\R^m\times\R^n\, | \, g(x,z,0)=0 \right\},
}
which serves as the phase space of the CDE \cref{intro:cde1} and as the set of equilibrium points of the layer equation \cref{intro:layer}. In the latter context, it is useful to recall the concept of Normally Hyperbolic Invariant Manifold (NHIM).

\begin{definition}[Normally Hyperbolic Invariant Manifold] Consider a slow-fast system given by a vector field of the form
\eq{X_\ve=\ve f(x,z,\ve)\parc{x}+g(x,z,\ve)\parc{z}.
}

The associated slow (invariant) manifold $S=\left\{ g(x,z,0)=0\right\}$ is said to be normally hyperbolic if each point of $S$ is a hyperbolic equilibrium point of $X_0$.
\end{definition} 

NHIMs are relevant in the context of the geometric study of slow-fast systems, see for example \cite{Fenichel}. It is known that compact NHIMs persist under $\mathcal C^1$ small perturbation of the vector field \cite{Jones,Kaper}. In the particular context presented above, a normally hyperbolic compact subset of the slow manifold $S$  persists as an invariant manifold of the slow-fast system $X_\ve$. We show in \cref{fig:intro1}  a schematic of the previous description.

\begin{figure}[htbp]\centering
  \includegraphics{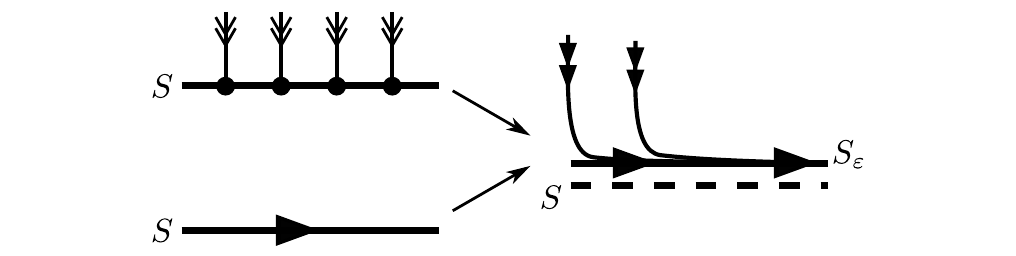}
  \caption{A schematic representation of the persistence of a NHIM under the perturbation of the corresponding vector field. $S$ denotes the slow manifold. Left-above: $S$ is a set of hyperbolic equilibrium points of the layer equation. Left-below: $S$ is the phase space of the constrained equation. Right: since $S$ is a NHIM, it persists as an invariant manifold $S_\ve$ under small perturbations of the vector field. }
  \label{fig:intro1}
\end{figure}

After this intruduction, we turn into the subject of this paper. Our goal is to understand the dynamics of a particular slow-fast system which has one fast and two slow variables given as

\eq{\label{intro:eqcusp}
X_\ve=\ve(1+\ve f_1)\parc{x_1}+\ve^2f_2\parc{x_2}-\left( z^3+x_2z+x_1+\ve f_3 \right)\parc{z},
}

where the functions $f_i=f_i(x_1,x_2,z)$, for $i=1,2,3$, are smooth and vanish at the origin. The corresponding slow manifold is defined by
\eq{
  S=\left\{ (x_1,x_2,z)\in\R^3 \,| \, z^3+x_2z+x_1=0  \right\}.
}

\begin{remark}
  The slow manifold $S$ can be regarded as the critical set of the cusp (or $A_3$) catastrophe, which is given as \rm\cite{Arnold_singularities,Brocker}\eq{\label{eqV}
  V(x_1,x_2,z)=\frac{1}{4}z^4+\frac{1}{2}x_2z^2+x_1z.
  }
\end{remark}

 We denote by $\Delta$ the set of points in $S$ at which $S$ is tangent to the fast direction, that is
 \eq{\label{eqD}
 \Delta=\left\{ (x_2,z)\in S \, | \, 3z^2+x_2=0\right\}.
 }

In other words, $\Delta$ is the set of degenerate critical points of \cref{eqV}. See figure \cref{fig:qual1} for a description of the slow manifold and the set $\Delta$.

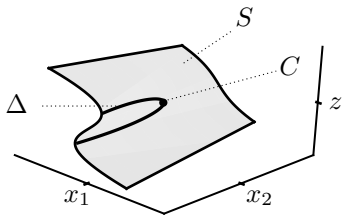
\begin{figure}[htbp]\centering
  \begin{tikzpicture}
    \pgftext{\includegraphics{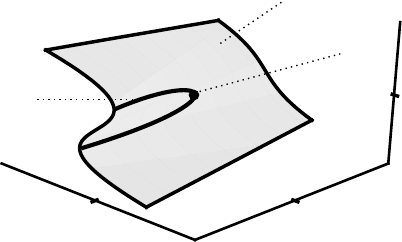}}
    \node at (1,1.4) {$S$};
    \node at (-2,0.25) {$\Delta$};
    \node at (1.6,.75) {$C$};
    \node at (2.2,.25) {$z$};
    \node at (1.2,-1) {$x_2$};
    \node at (-1.2,-1) {$x_1$};
  \end{tikzpicture}
  \caption{The manifold $S$ is two dimensional and can be defined as the critical set of the potential function $V(x_1,x_2,z)=\tfrac{1}{4}z^4 + \tfrac{1}{2}x_2z^2+x_1z$. The curve $\Delta$ is defined by the set of degenerate critical points of $V$. Geometrically, $B$ is the set of point of $S$ where $S$ is tangent to the fast direction, and $C$ denotes the cusp point.}
  \label{fig:qual1}
\end{figure}

Our interest in studying \cref{intro:eqcusp} is due to the fact that the origin $(x_1,x_2,z)=(0,0,0)$ is a \emph{non-hyperbolic equilibrium point} of $X_0$. This implies that a compact subset, around the origin, of the slow manifold $S$ is not a NHIM of $X_0$, and therefore, the Geometric Singular Perturbation Theory \cite{Fenichel,Jones,Kaper} is not enough.

\subsection{Motivation}

There have been several studies, e.g. \cite{Krupa3,Krupa20102841}, dealing with a SFS of the form
\eq{
X_\ve=\ve(1+f_1)\parc{x_1}-\left( z^2+x_1+ \ve h \right)\parc{z},
} 

whose slow manifold is the critical set of the fold catastrophe. The next natural step is to consider the following case in the Thom list \cite{Stewart1}, i.e., a slow-fast system induced by the cusp catastrophe. That is
\eq{\label{mot:eqcusp}
X_\ve=\ve(1+f_1)\parc{x_1}+\ve f_2\parc{x_2}-\left( z^3+x_2z+x_1+\ve f_3 \right)\parc{z}.
}

In \cite{BKK}, the system \cref{mot:eqcusp} is studied in a qualitative way. Here, however, we aim to refine the results by heavily using techniques from normal form theory. Moreover, we remark that the methods presented here are applicable to a larger class of slow-fast system given by
\eq{
  X_\ve=\ve(1+f_1)\parc{x_1}+\sum_{i=2}^{k-1}\ve f_i\parc{x_i}-\left( z^k+\sum_{j=1}^{k-1}x_jz^{j-1}-\ve f_k \right)\parc{z},
}

which is called (regular) $A_k$-SFS, see \cite{JardonThesis}.

\subsection{Statement}

We shall study the SFS

\eq{\label{s1}
  X_\ve=\ve (1+f_1)\parc{x_1}+\ve f_2\parc{x_2}-\left( z^3+x_2z+x_1 + \ve f_3 \right)\parc{z},
}

where the functions $f_i=f_i(x_1,x_2,z,\ve)$ are smooth. To avoid working with an $\ve$-parameter family of vector fields as \cref{s1}, it is customary to extend \cref{s1} by adding the trivial equation $\ve'=0$, and thus consider a smooth vector field in $\R^4$ which reads as
\eq{\label{s2}
  X=\ve (1+ f_1)\parc{x_1}+\ve f_2\parc{x_2}-\left( z^3+x_2z+x_1 + \ve f_3 \right)\parc{z} + 0\parc{\ve}.
}

We regard \cref{s2} as a perturbation of ``the principal part'' $F$ which is given as
\eq{\label{s3}
  F=\ve \parc{x_1}+0\parc{x_2}-\left( z^3+x_2z+x_1  \right)\parc{z} + 0\parc{\ve}.
}

 Note that in a qualitative sense, $F$ contains the essential elements of $X$. To state our main result, we first define the sections

\eq{
  \Sigma^-=\left\{ (x_1,x_2,z,\ve)\in\R^4\, |\, x_1=-x_1^{i}  \right\}\\
  \Sigma^-=\left\{ (x_1,x_2,z,\ve)\in\R^4\, |\, x_1=x_1^{f}  \right\},
} 

where $x_1^{i}>0$ and $x_1^{f}>0$ are arbitrarily large constants. For $\ve>0$ but sufficiently small, the sections $\Sigma^{-}$ and $\Sigma^+$ are transversal to the flow of $X_\ve$. Next, let $\Pi:\Sigma^{-}\to\Sigma^{+}$ be the Poincaré map induced by the flow of $X_\ve$. We shall prove the following.

\paragraph{{\bfseries{Transition along the cusp}} ({\rm see \cref{teo:main}})} {
Consider a slow-fast system given by \cref{s2}. Let $\Sigma^-$, $\Sigma^+$ and $\Pi:\Sigma^-\to\Sigma^+$ be defined as above. Then, we can choose coordinates in $\Sigma^-$ and in $\Sigma^+$ such that the map $\Pi$ reads as
\eq{
\Pi(X_2,Z,\ve)=(\tilde X_2,\tilde Z,\tilde \ve),
}

where $\tilde X_2=X_2+H(X_2,\ve)$ (with $H$ flat at $(X_2,\ve)=(0,0)$), $\tilde\ve=\ve$ and where

\eq{\label{s4}
  \tilde Z=\Phi(X_2,\ve)+Z\exp\left( -\frac{1}{\ve}(A(X_2,\ve)+\ve\Psi(X_2,Z,\ve)) \right),
}

where $A(X_2,0)>0$. Details of the functions $\Phi$, $A$, and $\Psi$ are given in \cref{teo:main}. In an heuristic way, this result is described in \cref{figcusp}.

\begin{figure}\centering
  \begin{tikzpicture}
    \pgftext{\includegraphics[scale=1.5]{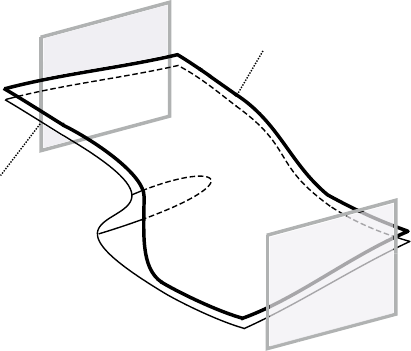}}
    \node at (1,2) {$S_\ve$};
    \node at (-3.3,-0.1) {$S$};
    \node at (-2,2.5) {$\Sigma^-$};
    \node at (3,-0.1) {$\Sigma^+$};
  \end{tikzpicture}
  \caption{Description of our main result. We may choose appropriate coordinates at the sections $\Sigma^-$ and $\Sigma^+$ under which the invariant manifold $S_\ve$ is given by $Z=0$. Moreover form \cref{s3} we have that all other trajectories starting at $\Sigma^-$ are exponentially attracted to the invariant manifold $S_\ve$. In this paper we provide quantitative information regarding this exponential contraction.}
  \label{figcusp}
\end{figure}
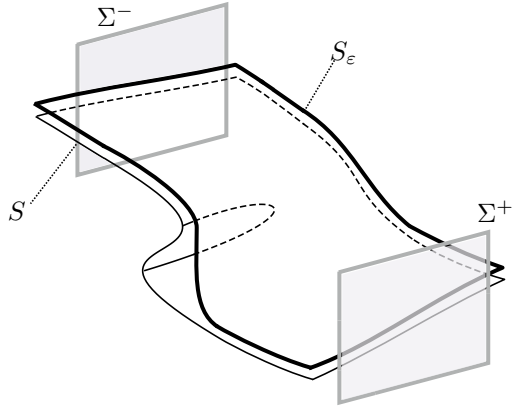

}

\subsection{Idea of the proof}

Our proof consists of two main steps.

\begin{enumerate}
  \item From \cite{Jardon2}, it is known that there exists a formal transformation bringing \cref{s2} into 
\eq{\label{s5}
  F=\ve \parc{x_1}+0\parc{x_2}-\left( z^3+x_2z+x_1  \right)\parc{z} + 0\parc{\ve}.
}

Then, by Borel's lemma \cite{Brocker}, the vector field $F$ can be realized as a smooth normal form $X^N=F+R$ of \cref{s2} and where $R$ is flat at $(x_1,x_2,z,\ve)=(0,0,0,0)$. See more details in \cref{sec:formal_nf}.

\item Based the previous normalization, next we use the geometric desingularization or blow up method (as introduced in \cite{DumRou2}) to study the flow of the normal form $X^N=F+R$. This is detailed in \cref{sec:GeomDes}.
\end{enumerate}

\begin{remark} With this document we aim at two goals:

\begin{enumerate}
  \item To refine the results of \cite{BKK}. This is, we do not only provide a qualitative description of the transition $\Pi$, but details on the differentiability of such a map is also presented.
  \item To prepare a framework for the geometric desingularization of $A_k$ slow-fast systems. These are a generalization of \cref{s2} given as
  \eq{
  X=\ve(1+ f_1)\parc{x_1} + \sum_{i=1}^{k-1}\ve f_i\parc{x_i} - \left( z^k+\sum_{j=1}^{k-1}x_jz^{j-1}+\ve f_k \right)\parc{z}+0\parc{\ve}.
  }
\end{enumerate}
  
\end{remark}

The rest of this document is arranged as follows: in \cref{sec:preliminaries} we provide a brief recollection of preliminary results that will simplify our later studies. Next, in \cref{sec:GeomDes} we pose our result and prove it by means of the geometric desingularization method and the results of  \cref{sec:preliminaries}. For readability purposes, many technicalities have been put in the appendix.

\section{Preliminaries of slow-fast systems}\label{sec:preliminaries}
%\subfile{Regular.tex}
In this section, we provide a number preliminary results that will be used later in \cref{sec:GeomDes}. First of all, we consider slow-fast systems along normally hyperbolic regions of the slow manifold. Afterwards, we recall a result from \cite{Jardon2} dealing with the normal form of \cref{s2}. We remark that we only consider SFS with one fast variable. Let us be more precise with the type of SFS that we shall study first.

\begin{definition} A slow-fast system is said to be (locally) regular around a point $p_0$, if its corresponding slow manifold is normally hyperbolic in a some neighborhood of $p_0$.
\end{definition}

\subsection{The slow vector field}

Let us consider a slow-fast system given by 
\eq{\label{sdi1}
  X_\ve = \sum_{i=1}^m \ve f_i(x,z,\ve)\parc{x_i}+H(x,z,\ve)\parc{z},
}

where $x\in\R^m$, $z\in\R$, and as usual $0<\ve\ll 1$. Furthermore, assume that $f(0,0,0)\neq 0$, $H(0,0,0)=0$ and $\tparcs{H}{z}(0,0,0)<0$. Thus $X_\ve$ is regular around $0\in\R^{m+2}$. The slow manifold associated to \cref{sdi1} is defined by 
\eq{
  S=\left\{ (x,z)\in\R^{m+1}\, | \, H(x,z,0)=0 \right\}.
}

From the defining assumptions of \cref{sdi1}, we have that $S$ is a NHIM in a neighborhood of the origin. By looking at the Jacobian of $X_\ve$ at $0$, it follows that there exists an $m+1$ dimensional a center manifold. Since $X$ is smooth, we can choose a $\Cl$  center manifold $\Wc$ for any $\ell<\infty$. The manifold $\Wc$ is given as a graph $z=\phi(x,\ve)$ where $\phi$ is a $\Cl$ function.

\begin{remark} Along the rest of the document we frequently make use of a finite class of differentiability. As it is customary in the present context, when we say that a manifold (or a map) is $\Cl$, we mean that such a manifold (or map) is $\ell$-differentiable for $\ell$ as large as necessary.
\end{remark}

The slow manifold $S$ is naturally given by the restriction $\Wc|_{\ve=0}=S$. Next, let us consider the vector field $\frac{1}{\ve}X_\ve(x,\phi,\ve)$. Since $\Wc$ is locally invariant, it follows that $\frac{1}{\ve}X_\ve$ is tangent to $\Wc$. Therefore the vector field
\eq{
  X^{slow}=\lim_{\ve\to 0}\frac{1}{\ve}X_\ve(x,\phi,\ve),
}

is tangent to $S$ at each point of $S$, and we call it \emph{the slow vector field}. We remark that the slow vector field $X^{slow}$ is only well defined whenever $\phi$ is invertible. 

\subsubsection{The slow divergence integral}\label{sec:sdi}

Associated to a regular slow-fast system and the corresponding slow vector field, the \emph{slow divergence integral} is defined here. For this, let $\Sigma^-$ and $\Sigma^+$ be two sections which are transversal to the flow of $X_\ve$ given by \cref{sdi1}. For $\ve\neq 0$ but sufficiently small, these sections are also transversal to the slow manifold $S$. Let $\gamma_\ve$ be a solution curve of $X_\ve$ chosen along a center manifold $\Wc$, thus $\gamma_\ve$ is transversal to the sections $\Sigma^-$ and $\Sigma^+$. In the limit $\ve=0$, the curve $\gamma_0$ is a curve along the slow manifold $S$. The idea now is to borrow the well-known divergence theorem \cite{Spivak} to get some sense on how the trajectories of $X_\ve$ are attracted to $S$ (recall that we made the assumption $\tparcs{H}{z}<0$). The divergence of $X_\ve$ (given by \cref{sdi1}) reads as
\eq{
  \Div X_\ve &= \parcs{H(x,z,\ve)}{z}+O(\ve).
} 

We can now take the integral of $\Div X_\ve$ along the orbit $\gamma_\ve$ of $X_\ve$ parametrized by the fast time $\tau$, we have
\eq{\label{sdi2}
  \int_{\gamma_\ve}\Div X_\ve  \, d\tau=\int_{\gamma_\ve} \left(  \parcs{H(x,z,\ve)}{z}+O(\ve) \right)d\tau.
}

The \emph{slow divergence integral} is defined by
\eq{
  I(t)=\int_{\gamma_0} \Div X_0 \, dt,
}

where $t$ is the slow time defined by the slow vector field $X^{slow}$. Our goal then is to relate the divergence integral \cref{sdi2} with $I$.

\begin{proposition} Under the assumptions made in this section, we have that
\eq{
\int_{\gamma_\ve}\Div X_\ve \, d \tau=\frac{1}{\ve}\left( I(t)+o(1)\right),
}

where $I(t)$ is the slow divergence integral.
  
\end{proposition}

\begin{proof}
  Recall that the slow vector field reads as $X^{slow}=\lim_{\ve\to 0}\frac{1}{\ve}X_\ve(x,\phi,\ve)$, where $\phi=\phi(x,\ve)$ is a $\Cl$ function. By our assumptions, the curve $\gamma_\ve$ is transversal to the sections $\Sigma^-$ and $\Sigma^+$ for $\ve$ small enough. Without loss of generality we can assume that $\gamma_\ve$ is parametrized by $x_1$. Then let $x_1^-$ and $x_1^+$ be defined by $\gamma_\ve(x_1^{-})=\gamma_\ve\cap\Sigma^-$ and $\gamma_\ve(x_1^{+})=\gamma_\ve\cap\Sigma^+$. Next,  the integral of the divergence of $X_\ve$ along $\gamma_\ve$ from $\Sigma^-$ to $\Sigma^+$ reads as

  \eq{
  \int_{\gamma_\ve}\Div X_\ve  \,  d\tau &=\frac{1}{\ve}\int_{x_1^-}^{x_1^+} \left(\parcs{H(x,z,0)}{z}+O(\ve)\right) \frac{dx_1}{f_1(x,z,0)+o(1)}\\
                    &=\frac{1}{\ve} \left(\int_{x_1^-}^{x_1^+} \parcs{H(x,z,0)}{z}\frac{dx_1}{f_1(x,z,0)} + o(1)\right)\\
                    &=\frac{1}{\ve} \left(\int_{\gamma_0}\Div X_0\, dt+o(1)\right),
  }

  where $t$ is the slow time induced by $X^{slow}$, which in coordinates means that $\frac{dx_1}{dt}=f_1$.
\end{proof}

Observe that the slow divergence integral is a first order approximation of the divergence along orbits of $X_\ve$. This will be useful when presenting our main result in \cref{sec:GeomDes}.

%%% --- Regular --- %%%

\subsubsection{Normal form and transition of a regular slow-fast system}\label{sec:nf_reg}

Now we consider the problem of finding a suitable normal form of a regular SFS. The following is a well-known result but we recall it here for completeness.

\begin{proposition}\label{prop:nf_reg}
  Consider a regular slow-fast system on $\R^{m+3}$ given by 
\eq{\label{eqr}
  X_\ve=\ve(1+ f_1)\parc{u}+\sum_{j=1}^m\ve g_j\parc{v_j}+H\parc{z},
}
where $(u,v_1,\ldots,v_m,z,\ve)\in\R^{m+3}$; where the functions $f_1=f_1(u,v,z,\ve)$ and $g_j=g_j(u,v,z,\ve)$, for $2\geq j\geq k-1$, are smooth and where the function $H=H(u,v,z,\ve)$ is smooth with $H(0,0,0,0)=0$ and $\tparcs{H}{z}(0,0,0,0)<0$. Then, the vector field $X$ is $\Cl$-equivalent to a normal form given by
\eq{
  X_{\ve}^N=\ve\parc{U} + \sum_{j=1}^m 0\parc{V_j}-Z\parc{Z},
}

where $\left\{ Z=0\right\}$ corresponds to a choice of the center manifold $\Wc$ of $X_\ve$.
\end{proposition}

\begin{proofof}{\cref{prop:nf_reg}} 
The first step is to divide the vector field $X$ by $1+f_1$. In a sufficiently small neighborhood of the origin this is a smooth equivalence relation. That is $Y=\tfrac{1}{1+f_1}X$ reads as
\eq{
Y= \ve\parc{u}+\sum_{j=1}^m\ve \tilde g_j\parc{v_j}+\tilde H\parc{z},
}

where $\tilde g_j$, for $2\geq j\geq k-1$, and $\tilde H$ are smooth with $\tilde H(0)=0$ and $\tparcs{\tilde H}{z}(0)<0$. Now we note that the origin of $\R^{m+3}$ is a semyhyperbolic equilibrium point with $(u,v,\ve)$ being center coordinates and $z$ being the hyperbolic coordinate. We can now use Takens-Bonckaert results on normal forms of partially hyperbolic vector fields \cite{Bonckaert1,Bonckaert2,Takens_partially}. Thus, there exists a $\Cl$ change of coordinates (maybe respecting some constraints if required) under which $Y$ is conjugated to
\eq{
  \bar Y=\ve\parc{U}+\sum_{j=1}^m\ve \bar G_j\parc{V_j}+\bar H Z\parc{Z},
}

where $\bar G_j=\bar G_j(U,V,\ve)$, for $2\geq j\geq k-1$, and $\bar H=\bar H(U,V,\ve)$ are $\Cl$ functions, and where $\left\{ Z=0\right\}$ corresponds to a choice center manifold which we denote by $\Wc$. We remark that in the vector field $\bar Y$, the functions $\bar G_j$ and $\bar H$ are independent of $Z$. Furthermore we have
\eq{
  \bar H(0,0,0)=\parcs{\tilde H}{z}(0,0,0,0)<0.
}

This means that in a small neighborhood of the origin $\bar Y$ can be divided by $|\bar H|$. In other words, $\bar Y$ is $\Cl$-equivalent to
\eq{
  \mathcal Y=\ve\mathcal G \parc{U}+\sum_{j=1}^m\ve \bar K_j\parc{V_j}-Z\parc{Z},
}

where $\mathcal G(0,0,0)\neq 0$ and $\bar K_j=\bar K_j(U,V,\ve)$, for $2\geq j\geq k-1$, are $\Cl$. Next, since $\Wc=\left\{ Z=0 \right\}$ is invariant under the flow of $\mathcal Y$, we can study the restriction $\mathcal Y|_{Z=0}$. This is
\eq{
  \mathcal Y|_{Z=0}=\ve\mathcal G \parc{U}+\sum_{j=1}^m\ve \bar K_j\parc{V_j}.
}

For $\ve\neq 0$, the vector field $\mathcal Y|_{Z=0}$ is regular because $\mathcal G(0,0,0)\neq 0$. Thus, by the flow-box theorem, there exists a change of coordinates, depending in a $\Cl$ way on $\ve$, under which $\mathcal Y|_{Z=0}$ can be written as
\eq{
  \ve\parc{U}+\sum_{j=1}^{m}0\parc{V_j}.
}

This implies that $\mathcal Y$ is $\Cl$-equivalent to
\eq{
  X_{reg}^N=\ve \parc{U}+\sum_{j=1}^m0\parc{V_j}-Z\parc{Z},
}

as stated in the proposition.

\end{proofof}

Motivated by \cref{prop:nf_reg} let us now discuss the dynamics of the vector field
\eq{\label{r1} X_{reg}^N=\ve \parc{U}+\sum_{j=1}^m0\parc{V_j}-Z\parc{Z}.
}

The slow manifold $S$, corresponding to the normal form \cref{r1}, is given by
\eq{
  S=\left\{ \ve=0, \, Z =0 \right\}.
}

Furthermore, we can parametrize the solution of \cref{r1} by $U$. Let us define the sections
\eq{\label{sec_reg}
  \Sigma^{-} &= \left\{ (U,V,Z,\ve)\in\R\times\R^m\times\R\times\R\, | \, U=U^- \right\}\\
  \Sigma^{+} &= \left\{ (U,V,Z,\ve)\in\R\times\R^m\times\R\times\R\, | \, U=U^+ \right\},
}

where $U^-<U^+$. The sections $\Sigma^{-}$ and $\Sigma^{+}$ are transversal to the manifold $S$ and therefore, for $\ve\neq 0$, are also transversal to the flow of \cref{r1}. Associated to these sections, we define the transition
\eq{\label{regtr}
  \Pi &:\Sigma^-\to\Sigma^+\\
  &(V,Z,\ve)\mapsto (\tilde V,\tilde Z,\tilde\ve).
}

To compute the component $\tilde Z$ we only need to integrate $\tfrac{dZ}{dU}=-\tfrac{1}{\ve}Z$. Then it follows that $\tilde Z=Z(T)$, where $T$ is the time to go from $\Sigma^-$ to $\Sigma^+$, which is $T=U_f-U_i$. Then it follows that
\eq{
  \tilde V&= V\\
  \tilde Z&=Z\exp\left(-\frac{1}{\ve}(U_f-U_i)\right)\\
  \tilde\ve&=\ve.
}

Observe the particular format of the transition $\Pi$. The $Z$ component is an \emph{exponential} contraction towards the center manifold $\left\{ Z=0 \right\}$. Maps with this characteristic appear frequently in our text and also in several other cases where slow-fast systems are studied. Therefore, in \cref{sec:Exp_trans} we discuss in a rather general way, the properties of such maps.

\subsection{Formal normal form of $A_k$ slow-fast systems}\label{sec:formal_nf}
%\subfile{formal_reduced.tex}
In this section we recall a normal form of the so-called $A_k$ slow-fast systems. A proof can be found in \cite{Jardon2}. This normalization is important since it eliminates many unwanted terms from the system being studied here.

\begin{definition}\label{def:AkSFS} Let $k\in\N$ with $k\geq 2$. An $A_k$ slow-fast system ($A_k$-SFS) is an ODE of the form
\eq{\label{formal1}
x_1' &= \ve(1+ f_1)\\
x_j' &= \ve f_j\\
z' &= -\left( z^k+\sum_{i=1}^{k-1} x_iz^{i-1} \right) + \ve f_k\\
\ve' &= 0,  
}

where $j=2,\ldots,k-1$, and where the functions $f_i=f_i(x_1,\ldots,x_{k-1},z,\ve)$, for $1\leq i\leq k$, are smooth.
\end{definition}

\begin{remark}\leavevmode
\begin{itemize}
  \item The system investigated in this work is an $A_3$-SFS. 

  \item The slow manifold associated to an $A_k$-SFS is defined by
  \eq{
  S=\left\{ (x,z)\in\R^k\, | \, z^k+\sum_{i=1}^{k-1} x_iz^{i-1}=0  \right\}.
  }

  The manifold $S$ can equivalently be defined as the critical set of an $A_k$ catastrophe \cite{Arnold_singularities}. Hence the name $A_k$-SFS.
\end{itemize}

\end{remark}

Locally, we can regard \cref{formal1} as $X=F+P$ where $F$ and $P$ are smooth vector fields of the form

\eq{\label{formalF}
F &= \ve\parc{x_1}+\sum_{j=2}^{k-1}0\parc{x_j}+g\parc{z}+0\parc{\ve}
}

and
\eq{\label{formalP}
P &=\sum_{i=1}^{k-1}\ve f_i\parc{x_i}+\ve f_k\parc{z}+0\parc{\ve},
}

respectively and where $g=-\left( z^k+\sum_{i=1}^{k-1} x_iz^{i-1} \right)$. We refer to $F$ as the ``principal part'' and to $P$ as the ``perturbation''. Briefly speaking we want to eliminate, via a change of coordinates, the perturbation. The procedure of normalizing the vector field $X$ is motivated by \cite{Sto10}, where normal forms of analytic perturbations of quasihomogeneous vector fields are investigated. The relevant result is the following

\begin{theorem}[Formal normal form \cite{Jardon2}]\label{prop:formal_nf} Let $k\geq 2$ and let $X=F+P$ be a smooth vector field where
\eq{
F=\ve\parc{x_1}+\sum_{i=2}^{k-1}0\parc{x_i}-\left( z^k+\sum_{j=1}^{k-1}x_jz^{j-1} \right)\parc{z}+0\parc{\ve}.
}

and where 
\eq{
  P=\sum_{i=1}^{k-1}P_i\parc{x_i}+P_k\parc{z}+0\parc{\ve},
}

where each $P_i=P_i(x_1,\ldots,x_{k-1},z,\ve)$ is a smooth function. Assume that the following conditions are satisfied
\begin{enumerate}
  \item $P_i(x_1,\ldots,x_{k-1},z,0)=0$,
  \item $\rho(\hat P_i)\geq 2k-i+1$,
\end{enumerate}

where $\hat P_i$ denotes the Taylor expansion of $P_i$ and $\rho(\hat P_i)$ is the quasihomogeneous order of the polynomial $\hat P_i$.
Then, there exists a formal diffeomorphism $\hat\Phi$ such that $\hat\Phi_*\hat X=F$. 
\end{theorem}

In words, \cref{prop:formal_nf} shows that $\hat X$ and $F$ are conjugated via $\hat\Phi$. It follows that, by Borel's lemma \cite{Brocker}, the formal vector field $\hat X^N=F$ can be realized as a smooth vector field $X^N=F+\tilde P$ where $\tilde P$ is \emph{flat} at $(x,z,\ve)=(0,0,0)$. This has important consequences in the geometric desingularization of an $A_3$-SFS, presented in the following section.

\section{Geometric desingularization of a slow-fast system near a cusp singularity}\label{sec:GeomDes}
\renewcommand{\a}{a}
\renewcommand{\b}{b}
\newcommand{\z}{z}

In this section we study an $A_3$ slow-fast system based on: a) the techniques introduced in \cref{sec:preliminaries} and in \cref{sec:Exp_trans}, and b) the blow up method. To simplify the notation, let us now write the $A_3$-SFS  as
\eq{\label{eq:cusp1}
X &=\ve(1+f_1)\parc{\a}+\ve f_2\parc{\b}-(z^3+\b z+\a+\ve f_3)\parc{z} + 0\parc{\ve},
}

where thanks to \cref{prop:formal_nf}, the smooth functions $f_i=f_i(a,b,z,\ve)$ are flat at the origin of $\R^4$. We invetigate the transition associated to \cref{eq:cusp1} between the sections
\eq{\label{def:cuspSectionsfar}
\Sigma^{-} &= \left\{ (\a,\b,z,\ve)\in\R^4\, | \, \a=-a^-, \, z>0 \right\}\\
\Sigma^{+} &= \left\{ (\a,\b,z,\ve)\in\R^4\, | \, \a=a^+, \, z<0 \right\},
}

where $a^->0$ and $a^+>0$ are arbitrarily large constants. However, since the trajectories of $X$ spend a long time along regular parts of $S$, it will be useful to define the ``entry'' and ``exit'' sections 
\eq{\label{def:cuspSections}
\Sigma^{\en} &= \left\{ (\a,\b,z,\ve)\in\R^4\, | \, \a=-\a_0, \, z>0 \right\}\\
\Sigma^{\ex} &= \left\{ (\a,\b,z,\ve)\in\R^4\, | \, \a=\a_0, \, z<0 \right\},
}

where $\a_0$ is a positive but sufficiently small constant, for reference see \cref{fig:rr}.

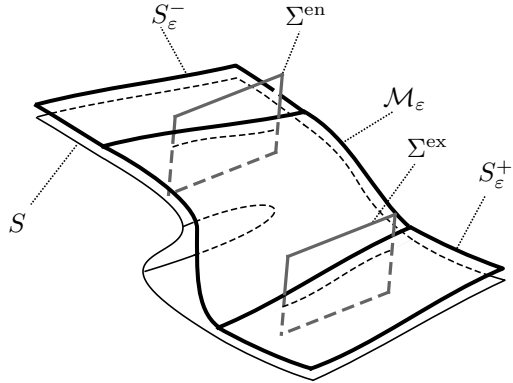
\begin{figure}[htbp]\centering
  \begin{tikzpicture}
    \pgftext{\includegraphics[scale=1.5]{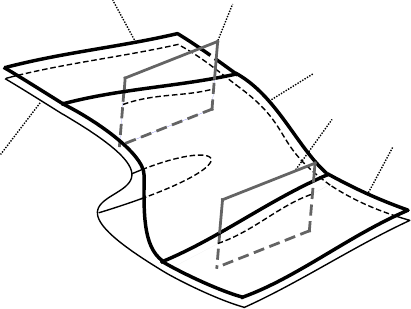}}
    \node at (-3.3,-.25) {$S$};
    \node at (-1.3,2.5) {$S_{\ve}^{-}$};
    \node at (3,0.4) {$S_{\ve}^{+}$};
    \node at (1.8,1.4) {$\mathcal{M}_{\ve}$};
    \node at (2.1,0.75) {$\Sigma^{\ex}$};
    \node at (.5,2.5) {$\Sigma^{\en}$};
  \end{tikzpicture}
  \caption{ Qualitative representation of the investigation performed in this section. The sections $\Sigma^{\en}$ and $\Sigma^{\ex}$ are arbitrarily close to the cusp point. On the other hand the sections $\Sigma^{-}$ and $\Sigma^{+}$ (not shown) are parallel to $\Sigma^{\en}$ and $\Sigma^{\ex}$ but far away from the cusp point. In a qualitative sense, we will construct an invariant manifold $\M_\ve$ and then extend it all the the way up to the sections $\Sigma^{-}$ and $\Sigma^{+}$. Our analysis aims for simplicity and thus depends extensively on the usage of normal forms. This, of course, makes our results coordinate-dependant.}
  \label{fig:rr}
\end{figure}

It will be clear from our analysis in the blow up space \cref{Ken} that the section $\Sigma^-$ needs to be partitioned as follows.

\begin{definition}[The inner layer and the lateral regions]\label{def:layers} Let $0<L<M<\infty$ be constants. The inner layer $\Sigma^{\inn}\subset\Sigma^-$ is defined as
  \eq{
    \Sigma^-\supset\Sigma^{\inn}=\left\{ (b,z,\ve)\in\Sigma^-\, | \, |b|<M\ve^{2/5} \right\}.
  }

  On the other hand, the lateral regions are defined as
  \eq{
    \Sigma^-\supset\Sigma^{+b}&=\left\{ (b,z,\ve)\in\Sigma^-\, | \, b>L\ve^{2/5} \right\}\\
    \Sigma^-\supset\Sigma^{-b}&=\left\{ (b,z,\ve)\in\Sigma^-\, | \, -b>L\ve^{2/5} \right\}.
  }
  Note that the set $\left\{\Sigma^{\inn},\Sigma^{+b}, \Sigma^{-b}  \right\}$ is an open cover of $\Sigma^-$, see \cref{fig:part0}

  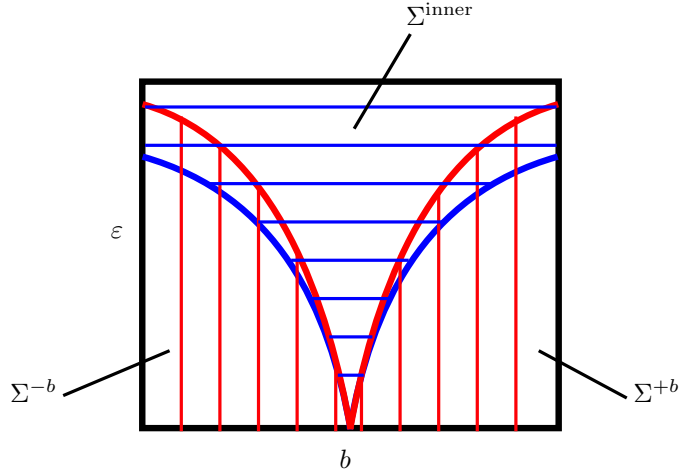
\begin{figure}[htbp]\centering
  \begin{tikzpicture}
    \pgftext{\includegraphics[scale=1.5]{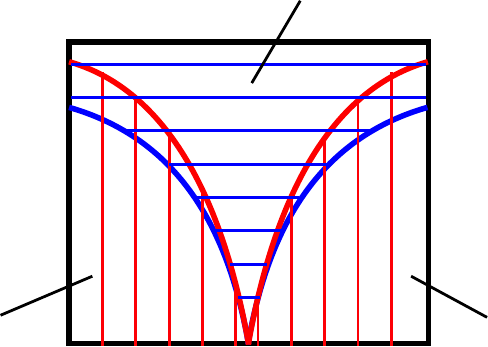}}
    \node at (1.25,2.9) {$\Sigma^{\inn}$};
    \node at (4.1,-2.1) {$\Sigma^{+b}$};
    \node at (-4.1,-2.1) {$\Sigma^{-b}$};
    \node at (0,-3) {$b$};
    \node at (-3,0) {$\ve$};
  \end{tikzpicture}\hfill
  % \begin{tikzpicture}
  %   \pgftext{\includegraphics[scale=1.5]{region_partition2.pdf}}
  %   \node at (0,-3) {$ $};
  % \end{tikzpicture}

  \caption{ The section $\Sigma^{-}$ needs to be partitioned into three subsections: the inner layer $\Sigma^{\inn}$ and the lateral regions $\Sigma^{+b}$, $\Sigma^{-b}$. From a qualitative point of view, these three layers correspond to three different types of trajectories: 1. Trajectories starting at $\Sigma^{\inn}$ pass close to the cusp point. Observe that $\lim_{\ve\to 0}(\Sigma^{\inn})=\left\{ b=0 \right\}$ and then corresponds to a solution of the associated $CDE$ passing exactly through the cusp point. 2. Trajectories starting at $\Sigma^{+b}$ pass sufficiently away from the cusp point along the regular side of the manifold $S$. 3. Trajectories starting at $\Sigma^{-b}$ pass sufficiently away from the cusp point along the folded side of the manifold $S$.  }
  \label{fig:part0}
\end{figure}
\end{definition}

We are now in position to present our main result. In the following theorem, we characterize the transition $\Pi:\Sigma^-\to\Sigma^+$ under a suitable choice of coordinates at the section $\Sigma^-$ and $\Sigma^+$. Furthermore, we give details on the differentiability of this map according to the cover of $\Sigma^-$, see \cref{def:layers}.

\begin{theorem}[Transition map of an $A_3$-SFS]\label{teo:main} Let $X$ be an $A_3$ slow-fast system. This is, $X$ is a vector field defined by
\eq{\label{eq:mainteo}
X=\ve(1+ f_1)\parc{\a}+\ve f_2\parc{\b}-\left(z^3+\b z+\a+\ve f_3\right)\parc{z} + 0\parc{\ve},
}

where each $f_i=f_i(\a,\b,z,\ve)$, $i=1,2,3$, is smooth. Let the sections $\Sigma^{-}$, $\Sigma^{+}$ be defined as above. Then we can choose suitable $\Cl$-coordinates $(B,Z,\ve)$ in $\Sigma^{-}$  and $\Cl$-coordinates $(\tilde B,\tilde Z,\tilde\ve)$ in $\Sigma^{+}$ such that the transition $\Pi:(B,Z,\ve)\mapsto (\tilde B,\tilde Z,\tilde\ve)$ is an exponential type map of the form
  \eq{\label{eq:expmainteo}
     \Pi(B,Z,\ve)=\left( B+h, \, \phi(B,\ve)+Z\exp\left( -\frac{A(B,\ve)+\Psi(B,Z,\ve)}{\ve} \right),\ve \right),
    }
    where $h$ is flat at the origin,  $A>0$ is $\Cl$, $\phi$ is $\Cl$-admissible with $\phi(B,0)=0$, and  $\Psi$ is $\Cl$-admissible with $\Psi(B,Z,0)=0$, see \cref{sec:Exp_trans} for the definition of $\Cl$-admissible. Moreover, we have the following properties of the function $A$, $\phi$ and $\Psi$.

\begin{enumerate}
  \item $-A(B,0)=I(B)$ where $I$ is the slow divergence integral associated to \cref{eq:mainteo}.
  \item Restricted to $(B,Z,\ve)\in\Sigma^{\inn}$, there are functions $\tilde \phi$ and $\tilde \Psi$ such that
   \eq{
   \phi(B,\ve) &=\tilde\phi\left(\mu,\ve^{1/5}\right)\\
      \Psi(B,Z,\ve) &= \tilde\Psi\left(|B|^{1/2},\ve^{1/5}, \ve\ln\ve, \mu,Z\right),
   }
   where $\tilde \phi$ and $\tilde \Psi$ are $\Cl$-functions with respect to monomials (see \cref{def:mon}) with $\mu=B\ve^{-2/5}$. Note that in this domain, $\mu$ is well defined in the sense that $\mu$ is bounded by a constant as $\ve\to 0$.
  \item Restricted to $(B,Z,\ve)\in\Sigma^{+b}$, there is a function $\tilde\Psi$ such that
   \eq{
   \phi(B,\ve) &= 0\\
      \Psi(B,Z,\ve) &= \tilde\Psi\left(|B|^{1/2},\ve^{1/5}, \ve\ln(|B|), \sigma,Z\right),
   }
   where $\tilde\Psi$ is a $\Cl$-function with respect to monomials (see \cref{def:mon}) with $\sigma=\ve|B|^{-5/2}$. Note that in this domain, $\sigma$ is well defined since $|B|>0$.
  \item  Restricted to $(B,Z,\ve)\in\Sigma^{-b}$, there are functions $\tilde \phi$ and $\tilde \Psi$ such that
   \eq{
   \phi(B,\ve) &=\tilde\phi\left(|B|^{1/2},\sigma\right)\\
      \Psi(B,Z,\ve) &= \tilde\Psi\left(|B|^{1/2},\ve^{1/5}, \ve\ln(|B|), \sigma\right),
   }
   where $\tilde \phi$ and $\tilde \Psi$ are $\Cl$-functions with respect to monomials (see \cref{def:mon}) with $\sigma=\ve|B|^{-5/2}$. Note that in this domain, $\sigma$ is well defined since $|B|>0$.
\end{enumerate}
\end{theorem}

\paragraph{Sketch of the proof.} { The first step is to recall \cref{prop:formal_nf}, which shows that $X$ is formally conjugate to 
\eq{
  F=\ve\parc{\a}+0\parc{\b}-\left(z^3+\b z+\a\right)\parc{z} + 0\parc{\ve}.
}

Next, by means of the Borel's lemma \cite{Brocker}, the vector field $F$ can be realized as a smooth vector field $X^N=F+\ve H$ where $H$ is flat at $(a,b,z,\ve)=(0,0,0,0)$. Thus, from now on, we only treat an $A_3$-SFS given as

\eq{
X=\ve(1+\ve \tilde f_1)\parc{\a}+\ve^2 \tilde f_2\parc{\b}-\left(z^3+\b z+\a+\ve \tilde f_3\right)\parc{z} + 0\parc{\ve},
}

where each $\tilde f_i=\tilde f_i(a,b,z,\ve)$ is flat at $(a,b,z,\ve)=(0,0,0,0)$.

Another important ingredient of the proof is the blow up technique, which is described in \cref{sec:blowup}. This method provides several local vector fields whose corresponding transitions are of exponential type, refer to \cref{sec:Exp_trans}. Later  all these local transitions are composed to produce an exponential type transition between the sections $\Sigma^-$ and $\Sigma^+$. Along the analysis of the local vector fields (in the blow up space) we will take advantage of the flatness of the higher order terms of $X$. The complete proof follows \cref{sec:blowup,Ken,Ke,Kex,Kpm} and is given in \cref{sec:proofofmain}.

Now, assuming that the transition $\Pi$ is of the form \cref{eq:expmainteo}, we can show that $A(B,0)$ is given by the slow divergence integral of $X$. For this, let us recall the Poincaré-Leontovich-Sotomayor formula \cite{DeMaesschalck2005}, which in general is given as follows. 

\begin{proposition} Let $X$ be a vector field on a manifold $M^n$ with a volume form $\Omega$. Let $\Sigma^-$ and $\Sigma^+$ be two open sections of $M$ and transverse to the flow of $X$. Let $\gamma_\ve$ be an orbit of $X$ along a center manifold $\Wc$ of $X$, starting at $p=\gamma_\ve\cap\Sigma^-$ and reaching $q=\gamma_\ve\cap\Sigma^+$ in finite time. Let $\Pi:\Sigma^-\to\Sigma^+$ be the transition map defined in a neighborhood of $p$. If $\psi^-:U\to\Sigma^-$ and $\psi^+:V\to\Sigma^+$, with $U\subset\R^{n-1}$ and $V\subset\R^{n-1}$, are coordinates in $\Sigma^-$ and in $\Sigma^+$ respectively, then
\eq{\label{aux:sdi}
  \det\left( D\left( (\psi^{+})^{-1}\circ\Pi\circ\psi^- \right) \right)(s^-)=\frac{\langle \Omega(p), D\psi^-(s^-)\times X(p) \rangle}{\langle \Omega(q), D\psi^+(s^+)\times X(p) \rangle}\exp\left( \int_{\gamma_\ve}\Div_{\Omega}X  \, d\tau \right),
}%

where $s^-=(\psi^-)^{-1}(p)$ and $s^+=(\psi^+)^{-1}(q)$. The integral is taken along the orbit $\gamma_\ve$ from $p$ to $q$ parametrized by the fast time $\tau$.
\end{proposition}

So we have the following.

\begin{proposition} \label{prop:slowdiv} Consider an $A_3$-SFS and assume that the transition $\Pi:\Sigma^-\to\Sigma^+$ is given by \cref{eq:expmainteo}. Then $-A(B,0)=I(B)$, where $I(B)$ is the slow divergence integral associated to the $A_3$-SFS.
\end{proposition}

\begin{proof} The only relevant component is $Z$, so denote by $\Pi_Z$ the $Z$-component of $\Pi$. The factor multiplying the exponential in \cref{aux:sdi} can be taken as a constant $C>0$. Then we have that \cref{aux:sdi} for the vector field of \cref{teo:main} reads as
\eq{\frac{\partial\Pi_Z}{\partial Z}= C\exp\left( \int_{\gamma_\ve}\Div_{\Omega}X  \, d\tau \right). }

Using the properties of the slow divergence integral described in \cref{sec:sdi}, and since $C\neq 0$, we have
\eq{\label{mainau0}
\frac{\partial\Pi_Z}{\partial Z} &= C\exp\left( \int_{\gamma_\ve}\Div_{\Omega}X \, d\tau \right)\\
&=\exp\left( \frac{1}{\ve}\left(\int_{\gamma_0}\Div X_0  \, dt+\ve\ln C+o(1) \right) \right)\\
&=\exp\left( \frac{1}{\ve}\left(I +O(\ve) \right) \right),
}

where $I$ is the slow divergence integral of $X$ along a curve in the slow manifold $S$ from $\Sigma^-$ to $\Sigma^+$. In principle, the limit $\ve\to 0$ of \cref{mainau0} is not well defined. However, according to our \cref{teo:main}, we have by differentiating \cref{eq:expmainteo} w.r.t. $Z$
\eq{\label{mainau1}
  \frac{\partial\Pi_Z}{\partial Z}=\exp\left( -\frac{A(B,\ve)+\ve \Psi(B,Z,\ve)}{\ve} \right).
}

Identifying \cref{mainau0} with \cref{mainau1} and taking the limit $\ve\to 0$ we have indeed that 
\eq{
  \lim_{\ve\to 0} (I+O(\ve))=\lim_{\ve\to 0} (-A(B,\ve)+\ve\Psi(B,Z,\ve)),
}

which shows the claim. Note that the slow divergence integral in the coordinates $(a,b,z)$ reads as
\eq{
  I(b)=\tilde I(b,\zeta^+)-\tilde I(b,\zeta^-),
}

where straightforward computations show that 
\eq{\tilde I(b,\zeta)=\frac{9}{5}\zeta^5+2\zeta^3b+b^2\zeta,
}

and where $\zeta^\pm$ is a constant defined by $(a^{\pm},b,\zeta^{\pm})\in\Sigma^{\pm}\cap S$. 

On the other hand, in normal coordinates and along regular parts of the slow manifold, the $A_3$-SFS can be written as (see \cref{sec:nf_reg})
\eq{
X(A,B,Z,\ve)=\ve\parc{A}+0\parc{B}-Z\parc{Z}+0\parc{\ve}.
}%

In these coordinates the slow divergence integral reads as
\eq{
  I=A^+-A^-,
}%
where $A^+$ and $A^-$ are the corresponding parametrizations of $\Sigma^+$ and $\Sigma^-$ (respectively) in the coordinates $(A,B,Z,\ve)$.
  
\end{proof}

}

\subsection{Blow-up and charts}\label{sec:blowup}

Let us briefly recall the blow up technique, for more details see e.g. \cite{DumRou1,DumRou2,Krupa1}. The vector field $X$ \cref{eq:cusp1} is quasihomogeneous \cite{Arnold_singularities,Jardon2}. Therefore, it is convenient to use the \emph{quasihomogeneous blow up}. This technique consists on performing a coordinate transformation defined by
\eq{\label{blowup}
\a =r^{3}\bar \a , \, \b  = r^{2}\bar \b, \,   z=r\bar z, \, \ve=r^{5}\bar\ve,
}

which is called the blow up map, and where $\bar \a^2+ \bar \b^2+\bar z^2+\bar\ve^2=1$ and $r\in[0,+\infty)$.   That is $(\bar a, \bar b, \bar z,\bar\ve,r)\in S^{3}\times\R^+$. Since $\ve\geq 0$, we can restrict the coordinates to $\bar\ve\geq 0$. Note that $S^{3}\times\left\{ 0 \right\}$ is mapped, via the blow up map \cref{blowup}, to the origin of $\R^{4}$. The powers or weights of the blow up map \cref{blowup} are obtained from the type of quasihomogeneity of $X$.

Let us denote by $\Phi(\bar \a,\bar\b, \bar z,\bar \ve)$ the blow up map \cref{blowup}. This map induces a smooth vector field $\tilde X$ on $S^{3}\times\R^+$  defined by $\Phi_*\tilde X=X$. It is often the case in which the vector field $\tilde X$ is degenerate along $S^{3}\times\left\{ 0 \right\}$. Then one defines another vector field $\bar X$ by $\bar X=\frac{1}{r^m}\tilde X$ for a well chosen positive integer $m$ so that $\bar X$ is non-degenerate along $S^{3}\times\left\{ 0 \right\}$. Since $r\in\R^+$, the phase portraits of $\tilde X$ and $\bar X$ are equivalent outside $S^{3}\times\left\{ 0 \right\}$, and therefore it is equally useful to study $\bar X$ instead of $\tilde X$. One obtains a complete description of the local flow of $X$ near the the cusp point by studying the flow of $\bar X$ for $(\bar \a,\bar\b, \bar z,\bar\ve,r)\in S^{3}\times [0,r_0)$ with $r_0>0$ sufficiently small.

For problems of dimension greater than $2$, performing computations in spherical coordinates becomes tedious. Therefore, it is more convenient to consider charts which parametrize hemispheres of the ball $S^{3}\times[0,r_0)$. In the present context, the useful charts are
\eq{\label{charts}
K_{\en}=\left\{ \bar \a=-1 \right\}, \; K_{\ex}=\left\{ \bar\a=1 \right\}, \; K_{\bar\ve}=\left\{ \bar\ve=1 \right\},\; K_{\pm}=\left\{ \bar\b=\pm 1 \right\} 
}

and we always keep $r\in[0,r_0)$. The previous setting is also known as directional blow up. A qualitative picture of the charts is given in \cref{fig:charts}.

Briefly speaking, our analysis goes as follows: first, we perform a local analysis on each chart given in \cref{charts}. Next, we compose (``glue'') the local results to provide a full description of the flow of $X$ \cref{eq:cusp1} in a small neighborhood of the cusp point. In this way, we construct an invariant manifold from $\Sigma^{\en}$ to $\Sigma^{\ex}$. Later we ``push away'' this invariant manifold all the way up to the sections $\Sigma^-$ and $\Sigma^+$ along regular parts of the slow manifold $S$. 

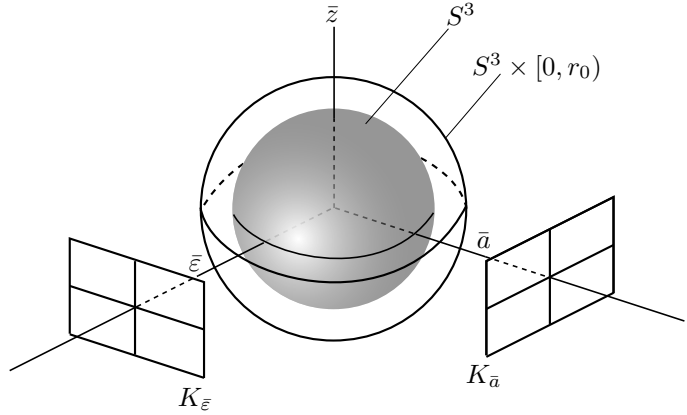
\begin{figure}[htbp]\centering
\begin{tikzpicture}
\pgftext{\includegraphics[]{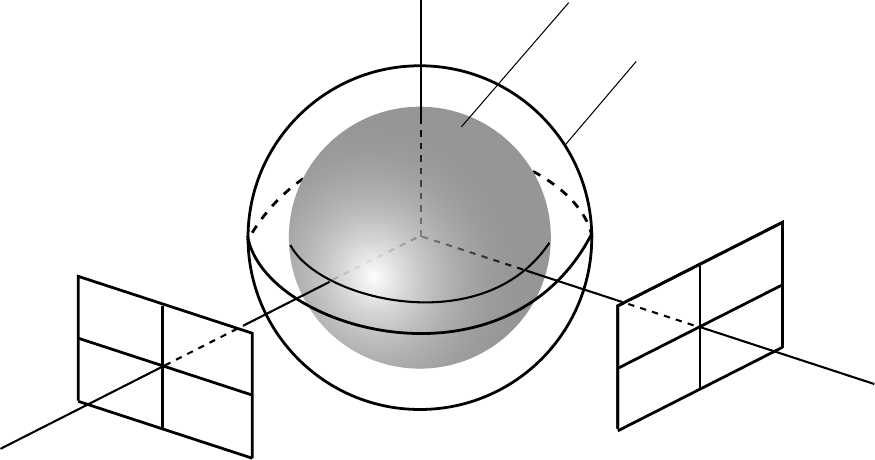}}  
\node at (2.5,1.8) {$S^3\times [0,r_0)$};
\node at (1.5,2.5) {$S^3$};
\node at (-.2,2.5) {$\bar z$};
\node at (-2,-0.8) {$\bar \ve$};
\node at (1.8,-0.5) {$\bar a$};
\node at (-2,-2.6) {$K_{\bar \ve}$};
\node at (1.8,-2.3) {$K_{\bar a}$};
\end{tikzpicture}
\caption{The blow up space and the charts. Each chart $K_\ell$ parametrizes a region of the ball $S^3\times [0,r_0)$. A local analysis in the charts provides a full picture of the dynamics of the vector field $\bar X$.}
\label{fig:charts}
\end{figure}

 To avoid confusion of the coordinates we adopt the following notation. Any object $O$ defined in the chart $K_{\en}$ is denoted by $O_1$. Similarly any object defined in the chart $K_{\ex}$ is denoted by $O_3$. Finally, an object $O$ defined in either of the charts $K_{\bar\ve}$ or $K_{\pm}$ is denoted by $O_{2}$.

%% =============== %%
%% =============== %%
%%    CHART Ken    %%
%% =============== %%
%% =============== %%

\subsection{Analysis in the chart $K_{\en}$}\label{Ken}
%\subfile{Ken.tex}
  Taking into account our notation convention, the blow-up map in this chart is given by
\eq{\label{blowupKen}
a=-r_1^3, \; b=r_1^2b_1, \; z=r_1^3z_1, \; \ve=r_1^5\ve_1.
}
\renewcommand{\r}{r_1}
\newcommand{\x}{a_1}
\newcommand{\y}{b_1}
\renewcommand{\z}{z_1}
\newcommand{\e}{\ve_1}
\renewcommand{\Wc}{\mathcal{W}_1^{^C}}

The corresponding vector field in this chart (after multiplication by $3$) has the form
\eq{X_{\en}: \begin{cases}
\r' &= - \e\r\left( 1+ \tilde f_1\right)\\
\y' &= 2 \e\y\left( 1+ \tilde f_1\right)+\r^6\e^2\tilde f_2\\
\z' &= -3\left( \z^3+\y\z-1-\frac{1}{3}\e\z \right)+\r^2\e\tilde f_3\\
\e' &= 5\e^2\left( 1+ \tilde f_1\right)
\end{cases}
}
where the functions $\tilde f_i=f_i(\r,\y,\z,\e)$ are flat along $\r=0$, recall that $S^3\times\left\{ r=0 \right\}\mapsto 0\in\R^4$ via the blow up map. We study a transition $\Pi_1:\Delta_{1}^{\en}\to\Delta_1^{\ex}$ where
\eq{
\Delta_{1}^{\en} &= \left\{ (\r,\y,\z,\e)\in\R^4 \, | \, \r=r_0, \e<\delta, \, \z>0 \right\}\\
\Delta_1^{\ex} &= \left\{ (\r,\y,\z,\e)\in\R^4 \, | \, \e=\delta, \r<r_0 \right\},
}
where $r_0$ and $\delta$ are sufficiently small positive constants. 

\begin{remark}
  The section $\Delta_1^{\en}$ corresponds to $\Sen$ in the blow-up space, that is $\Sen=\Phi(\Delta_1^{\en})$, where $\Phi$ is the blow-up map \cref{blowupKen}. This implies that trajectories of $X$ crossing $\Sen$ correspond to trajectories of $X_{\en}$ crossing $\Delta_1^{\en}$.
\end{remark}

Before going any further, let us provide a qualitative description of $X_{\en}$ as in \cite{BKK}. This process can be repeated, following similar arguments, in all the local charts; however, for brevity we only detail it for the current one.

\paragraph{Qualitative description of the flow of $X_{\en}$}{

The subspaces $\left\{ \r=0 \right\}$, $\left\{ \e=0 \right\}$ and $\left\{ \r=0 \right\}\cap\left\{ \e=0 \right\}$ are invariant. Therefore, it is useful to study the flow of $X_{\en}$ restricted to the aforementioned subspaces.

\begin{description}[leftmargin=0cm]
\item[Restriction to $\left\{ \r=0 \right\}\cap\left\{ \e=0 \right\}$.] In this space $X_{\en}$ is reduced to 
\eq{\label{eq:K_en1}
\y' &   = 0\\
\z' &   = -3\left( \z^3+\y\z-1 \right).\\
}

The set 
\eq{\label{gamma1}
\gamma_1 = \left\{ (\y,\z)\, | \,  \z^3+\y\z-1=0  \right\}
} 

is a curve of equilibrium points. The phase portrait of \eqref{eq:K_en1} is shown in figure \ref{fig:Ken1}.

\begin{figure}[htbp]\centering
\begin{tikzpicture}
\pgftext{\includegraphics[scale=1]{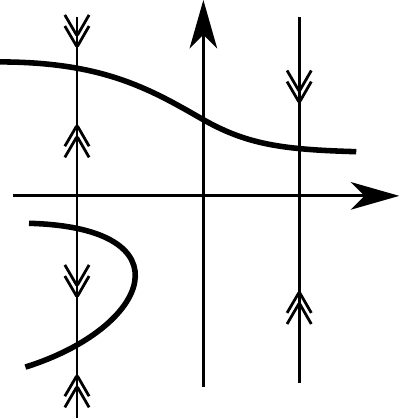}}
\node at (2.2,0.2) {$b_1$};
\node at (0.1,2.3) {$z_1$};
\end{tikzpicture}
\caption{The phase portrait of $X_{\en}$ restricted to the invariant space $\left\{ \r=0 \right\}\cap\left\{ \e=0 \right\}$. The shown curve is $\gamma_1$ and it comprises a set of equilibrium points. Note that locally, all trajectories with initial condition $z_1(0)>0$ are attracted to $\gamma_1|_{\left\{\z>0\right\}}$.}
\label{fig:Ken1}
\end{figure}

\begin{remark} All the trajectories of \cref{eq:K_en1} restricted to an initial condition $z_0>0$ are attracted to the curve $\gamma_1|_{z_1>0}$. Furthermore, due to our definition of $\Delta_1^{\en}$, we are interested \emph{only} in trajectories satisfying this initial condition. Thus, from now on, we restrict our analysis to the subspace $\left\{ \z>0 \right\}$.
  
\end{remark}

\item[Restriction to $\left\{ \e=0 \right\}$.] In this space $X_{\en}$ is reduced to
\eq{\label{eq:K_en2}
\r' &   = 0\\
\y' &   = 0\\
\z' &   = -3\left( \z^3+\y\z-1\right).\\
}

The set $\Gamma_1 = \left\{ (\r,\y,\z)\, | \,  \z^3+\y\z-1=0  \right\}$ is a surface of equilibrium points given by $\Gamma_1=(r_1,\gamma_1)$. Since $\r'=0$, the phase space of \cref{eq:K_en2} is foliated by two dimensional leaves in which the flow looks like \cref{fig:Ken1}.

\item[Restriction to $\left\{ \r=0 \right\}$.] In this space $X_{\en}$ is reduced to
\eq{\label{eq:K_en3}
\y' &   = 2\e\y  \\
\z' &   = -3\left( \z^3+\y\z-1-\frac{1}{3}\e\z \right) \\
\e' & = 5\e^2,
}

Once again, the set $\gamma_1 = \left\{ (\y,\z,\e)\, | \, \e=0, \,  z>0,\,  \z^3+\y\z-1=0  \right\}$ is a curve of equilibrium points. The Jacobian of \cref{eq:K_en3} evaluated along $\gamma_1$ shows that, for small enough $\e$, there exists an invariant center manifold that passes through $\gamma_1$. Furthermore, the non-zero eigenvalue corresponding to the $z$-direction is negative along $\gamma_1$. The phase portrait of \eqref{eq:K_en3} is shown in figure \ref{fig:Ken3}.\\

\begin{figure}[htbp]\centering
\begin{tikzpicture}
\pgftext{\includegraphics[scale=1]{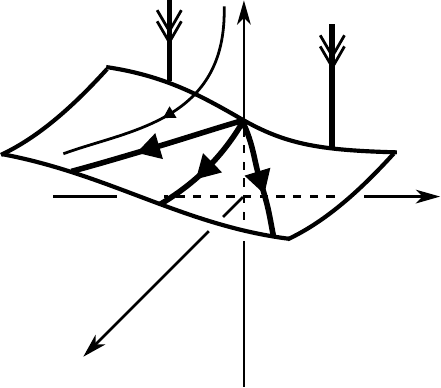}}
\node at (2.5,0) {$b_1$};
\node at (0.2,2.1) {$z_1$};
\node at (-1.5,-1.6) {$\ve_1$};
\end{tikzpicture}

\caption{Phase portrait of \eqref{eq:K_en3} restricted to $\z>0$. The shown surface is an invariant center manifold, which is attracting in the $z_1$-direction. }
\label{fig:Ken3}
\end{figure}

Observe that the $\y$ and the $\e$ directions are expanding. It is important to know the relation between such two expanding variables. We have
\eq{
\frac{d\y}{d\e}=\frac{2}{5}\frac{\y}{\e},
}

which has the solution
\eq{
\y=\y^*\left( \frac{\e}{\e^*} \right)^{2/5},
}

where $\y^*\leq\y$ and $\e^*\leq\e$ are the initial conditions, that is $(\y^*,\e^*)=(\y,\e)|_{\Delta_1^{\en}}$. It is important to look at the ratio of initial conditions $\tfrac{\y^*}{\left(\e^*\right)^{2/5}}$. This ratio tells us that $\y$ is bounded as $\e\to 0$ (and therefore as $\e^*\to 0$) if and only if $\y^*\in O\left( \left(\e^*\right)^{2/5} \right)$. In other words, if the initial condition $\y^*$ is not of order $O((\e^*)^{2/5})$ then the value of $\y$ at $\Delta_1^{\ex}$ blows up as $\e^*\to 0$. This leads us to partition the section $\Delta_1^{\en}$ into three open regions as follows.

\eq{
\Delta_1^{\en,\inn} &= \Delta_1^{\en}|_{|\y|<M\e^{2/5}}\\
\Delta_1^{\en,\y} &= \Delta_1^{\en}|_{\y>K\e^{2/5}}\\
\Delta_1^{\en,-\y} &= \Delta_1^{\en}|_{-\y>K\e^{2/5}},
}

where $0<K<M<\infty$. Observe that the open sets $\Delta_1^{\en,\inn}$, $\Delta_1^{\en,\y}$ and $\Delta_1^{\en,-\y}$ form an open cover of $\Delta_1^{\en}$. Accordingly, these sets induce an open cover of the entry section $\Sigma^{\en}$ via the blow up map \cref{blowupKen}. See \cref{fig:region_partition} for a representation of the aforementioned partition.\bigskip

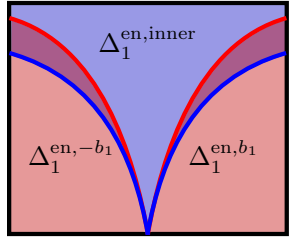
\begin{figure}[htbp]\centering
\begin{tikzpicture}\node at (-4,0){
\pgftext{\includegraphics[scale=1]{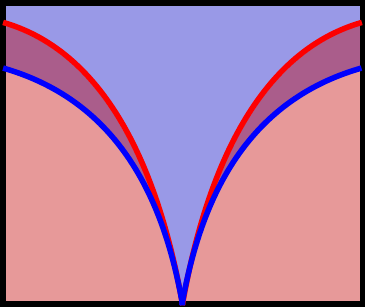}}}; 
\node at (-4,1) {$\Delta_1^{\en,\inn}$};
\node at (-5,-.5) {$\Delta_1^{\en,-\y}$};
\node at (-3,-.5) {$\Delta_1^{\en,\y}$};
\node at (0,-1.5) {$ $};
%\pgftext{\includegraphics[scale=1]{region_partition2.pdf}}
\end{tikzpicture}

\caption{Partition of $\Delta_1^{\en}$. Trajectories crossing through $\Delta_1^{\en,\e}$ corresponding to the inner wedge area, have a continuation on the chart $K_{\bar\ve}$. On the other hand, outside $\Delta_1^{\en,\e}$ we must consider the lateral regions $\Delta_1^{\en,\y}$ and $\Delta_1^{\en,-\y}$.  }
\label{fig:region_partition}
\end{figure}

Based on the partition of the entry section $\Delta_1^{\en}$, we define three transitions as follows
\eq{\label{trKen}
\Pi_1^{\inn}&:\Delta_1^{\en,\inn} \to \Delta_1^{\ex}\\
\Pi_1^{+\y}&:\Delta_1^{\en,+\y} \to \Delta_1^{\ex,+\y}\\
\Pi_1^{-\y}&:\Delta_1^{\en,-\y} \to \Delta_1^{\ex,-\y},
}

where 

\eq{
\Delta_1^{\ex} &= \left\{ (\r,\y,\z,\e)\in\R^4 \, | \, \e=\delta, \r<r_0 \right\},\\
\Delta_1^{\ex,\pm\y} &=\left\{ (\r,\y,\z,\e)\in\R^4 \, | \, \y=\pm\eta, \r<r_0 \right\}.
}

To finish with the qualitative description, note that there exists a (non-unique) $3$-dimensional center manifold $\Wc$, which is shown to exist by evaluating the Jacobian of $X_{\en}$ all along the surface 

\eq{
\Gamma_1=\left\{ (\r,\y,\z,\e)\, | \, \e=0, \, \z>0 \,\z^3+\y\z-1=0  \right\}.
}

Moreover, by the analysis provided above, the center manifold $\Wc|_{\z>0}$ is attracting for $\e$ small enough. Note that $\Wc|_{\e=0}=\Gamma_1$. This means that $\Wc$ can be interpreted as a perturbation of the slow manifold $S$, written in the coordinates of the current chart. See \cref{fig:Ken2} for a representation of the previous exposition.

\begin{figure}[htbp]\centering
\begin{tikzpicture}
  \pgftext{\includegraphics{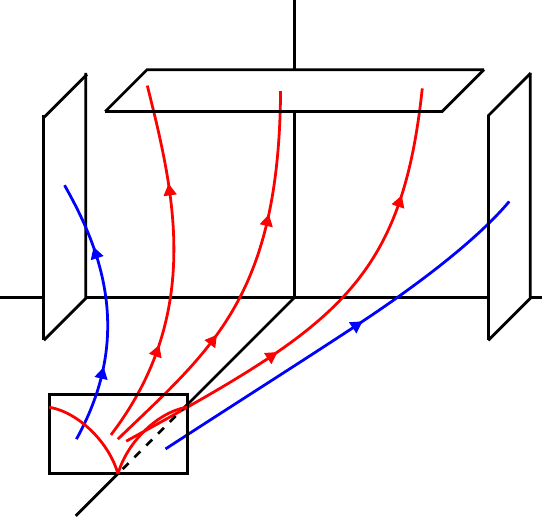}}
  \node at (3,-.5) {$\y$};
  \node at (-2,-2.8) {$\r$};
  \node at (0.25,2.8) {$\e$};

  \node at (-1,-2.4) {$\Delta_1^{\en}$};
  \node at (3.25,1) {$\Delta_1^{\ex,+\y}$};
  \node at (-3.,1) {$\Delta_1^{\ex,-\y}$};
  \node at (1,2.15) {$\Delta_1^{\ex,\e}$};
\end{tikzpicture}
\caption{Phase portrait of the trajectories of $X_{\en}$ depending on their initial condition. If the trajectories satisfy the estimate $y\in O(\ve^{2/5})$, then they arrive to $\Delta_1^{\ex,\ve_1}$ in finite time. If the estimate $y\in O(\ve^{2/5})$ is not satisfied, then we must choose one of the outgoing sections $\Delta_1^{\ex,\pm b}$ in order to have a well defined transition map. }
\label{fig:Ken2}
\end{figure}
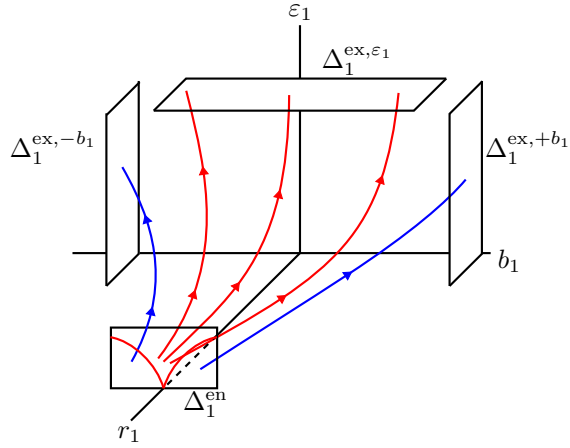
\end{description}

}%end of paragraph of the qualitative description.

Let us recall that the vector field $X_{\en}$ is of the form

\eq{\label{eq:Ken_original} X_{\en}: \begin{cases}
\r' &= - \e\r\left( 1+ \tilde f_1\right)\\
\y' &= 2 \e\y\left( 1+ \tilde f_1\right)+\r^6\e^2\tilde f_2\\
\z' &= -3\left( \z^3+\y\z-1-\frac{1}{3}\e\z \right)+\r^2\e\tilde f_3\\
\e' &= 5\e^2\left( 1+ \tilde f_1\right)
\end{cases}
}

We now proceed to describe the transitions $\Pi_1$ given by \cref{trKen}. For this, first we write \cref{eq:Ken_original} in a suitable normal form. Next, based on this normal form, we compute the corresponding transition. \bigskip

First of all, let us move the origin to the point $(\r,\b,\z,\e)=(0,0,1,0)$. This is done by defining a new variable $\zeta_1$ by $\zeta_1=z_1-1$. With this variable we have a new local vector field $Y_{\en}$ which is defined by

\eq{\label{eq:Ken_normal0}
  Y_{\en}:\begin{cases}
    \r' &= - \e\r\left( 1+ \tilde f_1\right)\\
    \y' &= 2 \e\y\left( 1+ \tilde f_1\right)+\r^6\e^2\tilde f_2\\
    \e' &= 5\e^2\left( 1+ \tilde f_1\right)\\
    \zeta_1' &= -3 G(\y,\e,\zeta_1)+\e\tilde h,
  \end{cases}
}
where $G(0,0,0)=0$ and $\tparcs{G}{\zeta_1}(0,0,0)=3$. Now, we want to write $Y_{\en}$ in a suitable normal form. From \cref{prop:nf_semihyp}, we know that $Y_{\en}$ is $\Cl$ equivalent to

\renewcommand{\y}{B_1}

\eq{\label{eq:Ken_normal1}X_{\en}^N: \begin{cases}
\r' &= - \e\r\\
\y' &= 2 \e\y \\
\e' &= 5\e^2\\
Z_1' &= -9(1+H_1(\r,\y,\e))Z_1,
\end{cases}
}

where $H_1$ is a $C^\ell$-function vanishing at the origin. This normal form $X_{\en}^{N}$ is convenient since the chosen center manifold $\Wc$ is now simply given by $\Wc=\left\{  Z_1=0 \right\}$. Furthermore, from the format of $X_{\en}^{N}$, it is evident the ``hyperbolic nature'' of the flow restricted to the center manifold: the restriction of $X_{\en}^N$ to the center manifold $\Wc$ has a simple structure, namely 

\eq{
X_{\en}^N|_{\Wc}: \begin{cases}
\r' &= - \e\r\\
\y' &= 2 \e\y \\
\e' &= 5\e^2.
\end{cases}
}

Note that for $\e\neq 0$, the vector field $\tfrac{1}{\e}X_{\en}^N|_{\Wc}$ is hyperbolic.\bigskip

\newcommand{\tr}{\tilde{r}_1}
\newcommand{\ty}{\tilde{B}_1}
\newcommand{\te}{\tilde{\ve}_1}
\newcommand{\tz}{\tilde{Z}_1}

The vector field $X_{\en}^N$ is of the form studied in \cref{prop:tr_semihyp}, therefore we have that the transition
\eq{
\Pi_1^{\inn}:(\y,\e,\z)\mapsto(\tr,\ty,\tz)
}

is of the form
\eq{
\tr &=r_0\left( \frac{\e}{\delta} \right)^{1/5}\\
\ty &=\y\left( \frac{\delta}{\e} \right)^{2/5}\\
\tz&=Z_1\exp\left( -\frac{9}{5\e}(1+\alpha_1\e\ln\e+\e G_1) \right),
}

where $\alpha_1=\alpha_1(r_0|\y|^{1/2},r_0\e^{1/5})$ and $ G_1= G_1(r_0|\y|^{1/2}, r_0\e^{1/5},\mu)$ where $\mu=\y\e^{-2/5}$. Recall that for this transition we have the condition $\y\in O(\e^{2/5})$ so $\mu$ is well defined.\bigskip

On the other hand, the transition
\eq{
\Pi_1^{\pm\y}:(\y,\e,Z_1)\mapsto(\tr,\te,\tilde Z_1)
}

is (see \cref{prop:tr_semihyp}) of the form
\eq{
\tr &=r_0\left( \frac{\y}{\eta} \right)^{1/2}\\
\te &=\e\left( \frac{\eta}{\y} \right)^{5/2}\\
\tz&=Z_1\exp\left( -\frac{9}{5\e}(1+\beta_1\e\ln(|\y|)+\e H_1) \right),
}

where $\beta_1=\beta_1(r_0|\y|^{1/2},r_0\e^{1/5})$ and $H_1=H_1(r_0|\y|^{1/2},r_0\e^{1/5},\sigma)$, where $\sigma=\e|\y|^{-5/2}$. Note that since $\y\notin O(\e^{2/5})$, $\sigma$ is well defined. We observe that the transitions $\Pi_1^{\e}$ and $\Pi_1^{\pm\y}$ are exponential type maps.

%% ============== %%
%% ============== %%
%%    CHART Ke    %%
%% ============== %%
%% ============== %%

\subsection{Analysis in the chart $K_{\bar\ve}$}\label{Ke}
%\subfile{Ke.tex}
Taking into account our notation convention, the blow-up map in this chart is given by

\eq{
a=r_2^3a_2, \; b=r_2^2b_2, \; z=r_2^3z_2, \; \ve=r_2^5.
}
\newcommand{\be}{\bar\ve}
\renewcommand{\r}{r_2}
\renewcommand{\x}{a_2}
\renewcommand{\y}{b_2}
\renewcommand{\z}{z_2}
\renewcommand{\e}{\ve_2}

Then, the blown up vector field reads as

\eq{ \label{eq:X2}X_{\be}:\begin{cases}
\r' &=0\\
\x' &=1+\tilde g_1\\
\y' &=r^6\tilde g_2\\
\z' &=-\left(\z^3+\y\z+\x\right)+\tilde g_3,
\end{cases}
}

where the function $\tilde g_i=\tilde g_i(\r,\x,\y,\z)$ are flat along $\r=0$. Note that in this chart $\r$ acts as a parameter and that the flow is regular. Furthermore, note that $X_{\be}$ is not a slow-fast system, but a regular vector field. From the equation $\x'=1+\tilde g_1$, we define the following ``entry'' and ``exit'' sections.
\eq{
\Delta_2^{\en,\be} &=\left\{ (\r,\x,\y,\z)\, | \, \x=-A_0, \, \z\geq 0 \right\},\\
\Delta_2^{\ex,\be} &=\left\{ (\r,\x,\y,\z)\, | \, \x=A_0, \, \z\leq 0 \right\}.\\
}
\renewcommand{\tr}{\tilde{r}_2}
\newcommand{\tx}{\tilde{a}_2}
\renewcommand{\ty}{\tilde{b}_2}
\renewcommand{\tz}{\tilde{z}_2}
\renewcommand{\te}{\tilde{\ve}_2}

Therefore, we define a transition $\Pi_2^{\be}$ as
\eq{
\Pi_2^{\be} : &\Delta_2^{\en,\be}\to\Delta_2^{\en,\be}\\
&(\r,\y,\z)\mapsto (\tr,\ty,\tz).
}

Since \cref{eq:X2} is regular, by the flow box theorem all trajectories starting at $\Delta_2^{\en,\be}$ arrive at $\Delta_2^{\ex,\be}$ in finite time. Moreover, the transition $\Pi_2^{\be}$ is a diffeomorphism and then, from \cref{eq:X2} we have that $\Pi_2^{\be}$ reads as
\eq{
  \Pi_2{\be}(\r,\y,\z) &= (\tr,\ty,\tz) \\
  &=(\r,\y+h_{\y},\phi_1(\r,\y)+ \phi_2(\r,\y)(1+\phi_3(\r,\y,\z))\z),
}

where the $\phi_i$'s are smooth functions. Observe that in this chart, the transition is not an exponential type map.

%% =============== %%
%% =============== %%
%%    CHART Kex    %%
%% =============== %%
%% =============== %%

\subsection{Analysis in the chart $K_{\ex}$}\label{Kex}
%\subfile{Kex.tex}

Taking into account our notation convention, the blow-up map in this chart is given by

\eq{
a=r_3^3, \; b=r_3^2b_3, \; z=r_3^3z_3, \; \ve=r_3^5\ve_3.
}
  \renewcommand{\r}{r_3}
  \renewcommand{\x}{a_3}
  \renewcommand{\y}{b_3}
  \renewcommand{\z}{z_3}
  \renewcommand{\e}{\ve_3}
  \renewcommand{\Wc}{\mathcal{W}_3^{^C}}

  \renewcommand{\tr}{\tilde{r}_3}
  \renewcommand{\tx}{\tilde{a}_3}
  \renewcommand{\ty}{\tilde{b}_3}
  \renewcommand{\tz}{\tilde{z}_3}
  \renewcommand{\te}{\tilde{\ve}_3}

  \renewcommand{\be}{\bar\varepsilon}
Then, the blown up vector field reads as

\eq{X_{\ex}: \begin{cases}
  \r' &=  \e\r\left( 1+ \tilde f_1\right)\\
  \y' &= -2 \e\y\left( 1+ \tilde f_1\right)+\r^6\e^2\tilde f_2\\
  \z' &= -3\left( \z^3+\y\z+1+\frac{1}{3}\e\z \right)+\r^2\e\tilde f_3\\
  \e' &= -5\e^2\left( 1+ \tilde f_1\right)
\end{cases}
}

where the function $\tilde f_i=\tilde f_i(\r,\y,\e,\z)$ are flat  along $\r=0$. Observe that the vector field $X_{\ex}$ resembles the vector field $X_{\en}$. Therefore, we have a similar behavior of the trajectories, the main difference is that in the case of $X_{\ex}$, there is one expanding ($\r$) and three contracting ($\y$, $\e$ and $\z$) directions. The flow of $X_{\ex}$ is obtained following similar arguments as for the flow of $X_{\en}$.\bigskip

From the fact that $X_{\ex}$ has three contracting and one expanding direction, we define the entry sections

\eq{
\Delta_{3}^{\en,\be} &=\left\{ (\r,\y,\e,\z)\, : \, \e=\delta,\, \z<0, \, \r<r_0\right\}\\
\Delta_{3}^{\en,+\y} &=\left\{(\r,\y,\e,\z)\, : \,\y=\eta ,\, \z<0, \, \r<r_0\right\}\\
\Delta_{3}^{\en,-\y} &=\left\{(\r,\y,\e,\z)\, : \,\y=-\eta,\, \z<0, \, \r<r_0 \right\},
}

where all the constants are positive and sufficiently small, and the exit section

\eq{
\Delta_{3}^{\ex} &=\left\{ (\r,\y,\e,\z)\, : \, \r=r_0, \, \z<0,\, \e<\delta, \, |\y|<\eta \right\}.
}

Then, accordingly, we define three transition maps as follows

\eq{\label{Kextrs}
\Pi_3^{\e} &:\Delta_{3}^{\en,\be} \to \Delta_{3}^{\ex}\\
&:(\r,\y,\z)\mapsto(\ty,\te,\tz)\\[2ex]
\Pi_3^{+\y} &:\Delta_{3}^{\en,+\y}\to \Delta_{3}^{\ex}\\
&:(\r,\e,\z)\mapsto(\ty,\te,\tz)\\[2ex]
\Pi_3^{-\y} &:\Delta_{3}^{\en,-\y}\to \Delta_{3}^{\ex}\\
&:(\r,\e,\z)\mapsto(\ty,\te,\tz).
}

\renewcommand{\y}{B_3}
\renewcommand{\ty}{\tilde{B}_3}
Now we proceed to write $X_{\ex}$ in a normal form just as we did with $X_{\en}$ in \cref{Ken}. Following \cref{prop:nf_semihyp} we have that $X_{\ex}$ is $\Cl$ equivalent to

\eq{X_{\ex}^N:\begin{cases}
  \r' &=  \e\r\\
  \y' &= -2 \e\y\\
  \e' &= -5\e^2\\
  Z_3' &= -9(1+H_3)Z_3,
\end{cases}
}

where $H_3=H_3(\r,\y,\e)$ is a $C^\ell$ function vanishing at the origin. Just as in the chart $K_{\en}$, there exists a three dimensional center manifold $\Wc$ associated to $X_{\ex}^N$ and which has been chosen such that $\Wc=\left\{ Z_3=0 \right\}$. Since $\r$ is the only expanding direction, we take as transition time $T_3=\ln\left( \tfrac{r_0}{\r} \right)$. This transition time is computed from the dynamics restricted to $\Wc$, that is, from the equation $\r'=\r$. In contrast to what happened in the chart $K_{\en}$, the time $T_3$ is well defined for all the three transitions $\Pi_3^{\e}$, $\Pi_3^{+\y}$ and $\Pi_3^{-\y}$. Following \cref{prop:tr_semihyp} we have

\eq{
  \ty &= \y \left( \frac{\r}{r_0} \right)^{2}\\
  \te &=\e \left( \frac{\r}{r_0} \right)^{5}\\
  \tilde Z_3 &= Z_3\exp\left( -\frac{9}{5\e}\left( \left( \frac{r_0}{\r} \right)^{5}-1+\alpha_3\e\ln\r+\e H_3 \right) \right),
}
where $\alpha_3=\alpha_3(\r|\y|^{1/2}, \r\e^{1/5})$ and $H_3=H_3(\r|\y|^{1/2}, \r\e^{1/5},\r)$. Therefore, by taking the definitions of the entry sections we have

\eq{
  \Pi_3^{\ve_3}(r_3,B_3,Z_3) &= \left( \y \left( \frac{\r}{r_0} \right)^{2}, \,  \delta \left( \frac{\r}{r_0} \right)^{5},\, Z_3\exp\left( -\frac{9}{5\delta}\left( \left( \frac{r_0}{\r} \right)^{5}-1+\alpha_3\delta\ln\r+\delta H_3 \right) \right) \right)\\
  \Pi_3^{\pm b_3}(r_3,\ve_3,Z_3) &= \left( \pm\eta \left( \frac{\r}{r_0} \right)^{2}, \,  \ve_3 \left( \frac{\r}{r_0} \right)^{5},\, Z_3\exp\left( -\frac{9}{5\ve_3}\left( \left( \frac{r_0}{\r} \right)^{5}-1+\alpha_3\ve_3\ln\r+\ve_3 H_3 \right) \right) \right).
}

Observe that these transitions are of exponential type.
%% =============== %%
%% =============== %%
%%    CHART Kb    %%
%% =============== %%
%% =============== %%

\subsection{Analysis in the charts $K_{\pm\bar b}$}\label{Kpm}
%\subfile{Kb.tex}
In this section we study the local flow at the charts $K_{+\bar b}$ and $K_{-\bar b}$. In a qualitative sense, these charts come into play when the initial condition $b_0=b|_{\Sigma^{\en}}$ does not satisfy the estimate $b_0\in O(\ve^{2/5})$. This implies that the corresponding trajectory passes away from the cusp point. The chart $K_{+\bar b}$ ``sees'' trajectories with initial condition $b|_{\Sigma^{\en}}>0$ while $K_{-\bar b}$ ``sees'' trajectories with initial condition $b|_{\Sigma^{\en}}<0$.

%% =============== %%
%% =============== %%
%%    CHART K+b    %%
%% =============== %%
%% =============== %%

\subsection*{Analysis in the chart $K_{+\bar{b}}$}
In this chart the blow-up maps reads
\eq{
  a=r_2^3a_2, \, b=r_2^2, \, z=r_2z_2, \, \ve = r_2^5\ve_2.
}

\renewcommand{\r}{r_2}
  \renewcommand{\x}{a_2}
  \renewcommand{\y}{b_2}
  \renewcommand{\z}{z_2}
  \renewcommand{\e}{\ve_2}
  \renewcommand{\Wc}{\mathcal{W}_2^{^C}}

\renewcommand{\tr}{\tilde{r}_2}
  \renewcommand{\tx}{\tilde{a}_2}
  \renewcommand{\ty}{\tilde{b}_2}
  \renewcommand{\tz}{\tilde{z}_2}
  \renewcommand{\te}{\tilde{\ve}_2}

Then we have that the blow-up vector field is given by
\eq{
  X_{+\bar{b}}:\begin{cases}
    \r' &= \e\bar f_r\\
    \x' &= \e(1+\bar f_{\x})+\e\bar g_{\x}\\
    \e' &= -\e\bar f_{\e}\\
    \z' &= -(\z^3+\z+\x)+\e\bar f_{\z}
  \end{cases}
}

where all the functions $\bar f_\ell$ are flat along $\left\{ \r=0 \right\}$. Observe that the set

\eq{\label{gamma2}
  \Gamma_2=\left\{ (\r,\x,\e,\z) \, |\, \e=0,\, \z^3+\z+\x=0 \right\}
}

is a NHIM of $X_{+\bar{b}}$. However, $X_{+\bar{b}}$ is not exactly a slow-fast system since $\e'\neq 0$, but the restriction of $X_{+\bar{b}}$ to $\left\{\r=0\right\}$ is indeed a slow-fast system. This restriction reads as
\eq{\label{Kb1}
  X_{+\bar{b}}|_{\left\{ \r=0 \right\}}:\begin{cases}
    \x' &= \e\\
    \e' &= 0\\
    \z' &= -(\z^3+\z+\x).
  \end{cases}
}

\begin{remark} The subspace $\left\{\r=0\right\}$ is invariant. Moreover, since $X_{+\bar{b}}$ is a flat perturbation of $X_{+\bar{b}}|_{\left\{ \r=0 \right\}}$, it is equally useful to study the restriction $X_{+\bar{b}}|_{\left\{ \r=0 \right\}}$. After all, by regular perturbation theory, their flows are equivalent.
\end{remark}

The slow manifold  of $X_{+\bar{b}}|_{\left\{ \r=0 \right\}}$ is defined by $\Gamma_2|_{\r=0}$ and is normally hyperbolic. Let us define the sections
\eq{\label{Kbsections}
  \Delta_2^{\en,+\y} &= \left\{ (\r,\x,\e,\z)\in\R^4\, | \, \x=-A_0 \right\}\\
  \Delta_2^{\ex,+\y} &= \left\{ (\r,\x,\e,\z)\in\R^4\, | \, \x=A_0 \right\}.
}

Accordingly, we study the transition
\eq{
  \Pi_2^{+\y}: & \Delta_2^{\en,+\y}\to\Delta_2^{\ex,+\y}\\
  &(\r,\e,\z)\mapsto (\tr,\te,\tz).
}

For a qualitative description of $X_{+\bar{b}}|_{\left\{ \r=0 \right\}}$ and the objects defined above see \cref{fig:Ky1}.

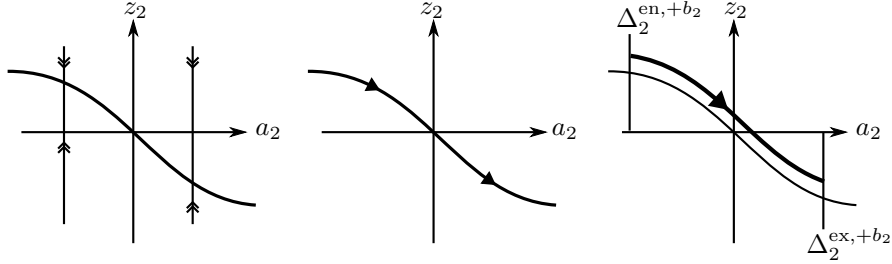
\begin{figure}[htbp]\centering
  \begin{tikzpicture}
    \pgftext{\includegraphics[scale=1]{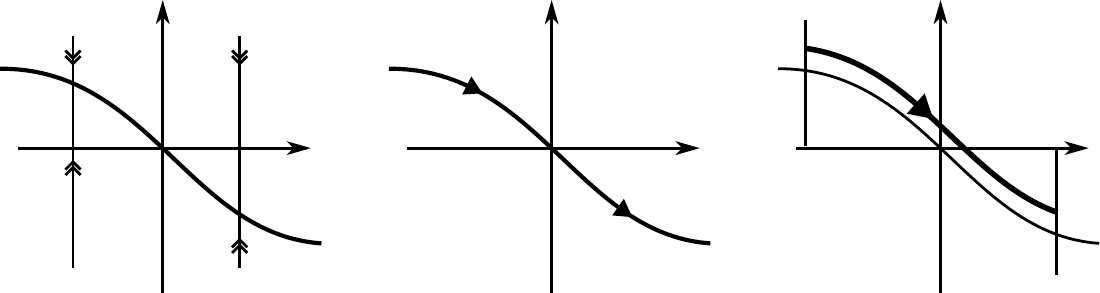}}
    \node at (1.75,0) {$\x$};
    \node at (0,1.6) {$\z$};

    \node at (1.75+4,0) {$\x$};
    \node at (0+4,1.6) {$\z$};

    \node at (1.75-3.9,0) {$\x$};
    \node at (0-3.9,1.6) {$\z$};

    \node at (3,1.5) {$\Delta_2^{\en,+\y}$};   
    \node at (5.5,-1.5) {$\Delta_2^{\ex,+\y}$};    
    \end{tikzpicture}
  \caption{Left: phase portrait of the corresponding layer equation of $X_{+\bar{b}}|_{\left\{ \r=0 \right\}}$. Center: phase portrait of the corresponding CDE of $X_{+\bar{b}}|_{\left\{ \r=0 \right\}}$. Right: Since the critical manifold is regular, by Fenichel theory we know that the manifold $\Gamma_2$ is perturbed to an invariant manifold $\Gamma_{2,\ve_2}$ which is at distance of order $O(\ve_2)$ from $\Gamma_2$. }
  \label{fig:Ky1}
\end{figure}

We know from \cref{sec:nf_reg} that for sufficiently small $\e$, there exists a $C^\ell$ change of coordinates that transforms $X_{+\bar{b}}|_{\left\{ \r=0 \right\}}$ into the vector field

\eq{Y^N:\begin{cases}
  \x' &= \e\\
  \e' &= 0\\
  Z_2' &= -Z_2,
\end{cases}
}

% where $\Lambda_2(0,0)>0$. In fact, since the spectrum of $X_{+\bar{b}}|_{\left\{ \r=0 \right\}}$ at the origin is $(0,0,-1)$, we have that $\Lambda_2(0,0)=1$. That is, $\Lambda_2$ can be written as $\Lambda_2=1+\alpha_2(\x,\e)$ for some $C^\ell$ function $\alpha_2$ vanishing at the origin. 
From the definition of the entry and exit sections \cref{Kbsections}, the time of integration is $T=2A_0$. To obtain the component $Z_2$ of the transition $\Pi_2^{+\y}|_{\left\{ \r=0 \right\}}$ we need to integrate

\eq{
  Z_2' = -\frac{1}{\e}Z_2,
}

and then $\tilde Z_2=Z_2(T)$. Therefore we have that after choosing a center manifold $\Wc$, the transition $\Pi_2^{+\y}$ reads as

\eq{
  \Pi_2^{+\y}(0,\e,Z_2) &=\left( 0,\e, Z_2\exp\left( -\frac{2A_0}{\e}\right)\right).
}

Note that $\Pi_2^{+\y}$ is an exponential type map.

%% =============== %%
%% =============== %%
%%    CHART K-b    %%
%% =============== %%
%% =============== %%

\subsection*{Analysis in the chart $K_{-\bar b}$}

In this chart the blow-up maps reads
\eq{
  a=r_2^3a_2, \, b=-r_2^2, \, z=r_2z_2, \, \ve = r^5\ve_2.
}

\renewcommand{\r}{r_2}
\renewcommand{\x}{a_2}
\renewcommand{\y}{b_2}
\renewcommand{\z}{z_2}
\renewcommand{\e}{\ve_2}
\renewcommand{\Wc}{\mathcal{W}_2^{^C}}

\renewcommand{\tr}{\tilde{r}_2}
\renewcommand{\tx}{\tilde{a}_2}
\renewcommand{\ty}{\tilde{b}_2}
\renewcommand{\tz}{\tilde{z}_2}
\renewcommand{\te}{\tilde{\ve}_2}

Then we have that the blow-up vector field is given by
\eq{
  X_{-\bar{b}}:\begin{cases}
    \r' &= -\e\bar f_r\\
    \x' &= \e(1+\bar f_{\x})+\e\bar g_{\x}\\
    \e' &= \e\bar f_{\e}\\
    \z' &= -(\z^3-\z+\x)+\e\bar f_{\z}
  \end{cases}
}

where all the functions $\bar f_\ell$ and $\bar g_{\x}$ are flat along $\left\{ \r=0 \right\}$. Observe that, as in the previous section, the subspace $\left\{ \r=0 \right\}$ is invariant. The restriction of $X_{-\bar{b}}$ to this subspace reads as
\eq{
  X_{-\bar{b}}|_{\left\{ \r=0 \right\}}:\begin{cases}
    \x' &= \e\\
    \e' &= 0\\
    \z' &= -(\z^3-\z+\x).
  \end{cases}
}

The flow of $X_{-\bar{b}}$ is a flat perturbation of the flow of $X_{-\bar{b}}|_{\left\{ \r=0 \right\}}$. Therefore, let us continue our analysis restricted to the invariant space $\left\{ \r=0 \right\}$.\bigskip

The manifold $\Gamma_2$, which is defined by
\eq{
  \Gamma_2=\left\{ (\r,\x,\e,\z) \, |\, \r=0,\, \e=0, \, \z^3-\z+\x=0 \right\}
}

is normally hyperbolic except at the two points $p_{\pm}=\pm\left( \frac{2}{3\sqrt{3}},\frac{1}{\sqrt{3}} \right)$. Let us define the sections
\eq{
  \Delta_2^{\en,-\y} &= \left\{ (\r,\x,\e,\z)\in\R^4\, | \, \x=-A_0 \right\}\\
  \Delta_2^{\ex,-\y} &= \left\{ (\r,\x,\e,\z)\in\R^4\, | \, \x=A_0 \right\},
}

where $A_0>0$ is a sufficiently large constant. We are interested in the transition
\eq{
  \Pi_2^{-\y}: & \Delta_2^{\en,-\y}\to\Delta_2^{\ex,-\y}\\
  &(\r,\e,\z)\mapsto (\tr,\te,\tz).
}

For a qualitative description of $X_{-\bar{b}}|_{\left\{ \r=0 \right\}}$ and the objects defined above see \cref{fig:Ky2}.

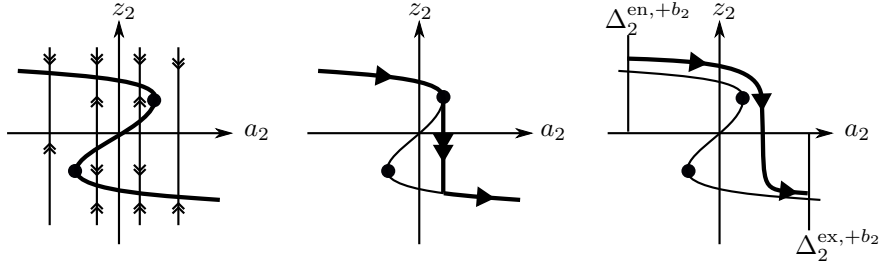
\begin{figure}[htbp]\centering
  \begin{tikzpicture}
    \pgftext{\includegraphics[scale=1]{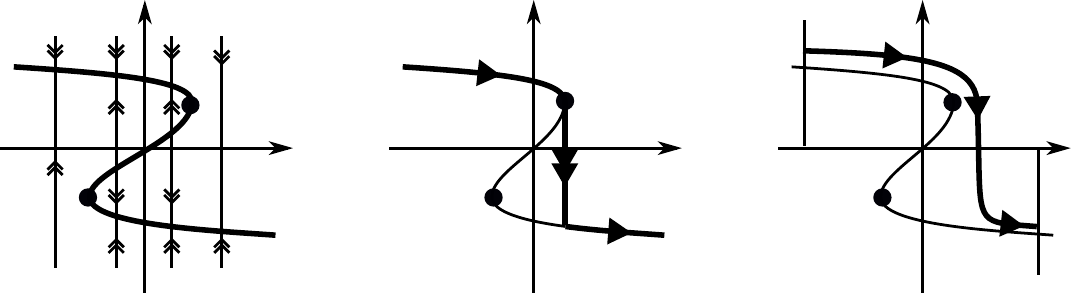}}
    \node at (1.75,0) {$\x$};
    \node at (0,1.6) {$\z$};

    \node at (1.75+4,0) {$\x$};
    \node at (0+4,1.6) {$\z$};

    \node at (1.75-3.9,0) {$\x$};
    \node at (0-3.9,1.6) {$\z$};

    \node at (3,1.5) {$\Delta_2^{\en,+\y}$};   
    \node at (5.5,-1.5) {$\Delta_2^{\ex,+\y}$}; 
    \end{tikzpicture}
  \caption{Left: phase portrait of the corresponding layer equation of $X_{-\bar{b}}|_{\left\{ \r=0 \right\}}$. Center: phase portrait of the corresponding CDE of $X_{-\bar{b}}|_{\left\{ \r=0 \right\}}$. Right: The expected perturbed invariant manifold obtained from the flow of the corresponding CDE and layer equation.}
  \label{fig:Ky2}
\end{figure}

Away from the fold points $p_\pm$, the manifold $\Gamma_2$ is regular and thus, Fenichel's theory applies. However, we need to take care of the transition near the fold point $p_+$. The local transition of a slow-fast system near a fold point is investigated in e.g. \cite{Krupa3}. However, in our current problem this transition is not essential. By this we mean that the passage through the fold point is seen as a flat perturbation of the trajectory along the stable branch of $\Gamma_2$. In a qualitative sense, this is due to the fact that the transition $\Pi_2^{-\y}$ goes along a large NHIM, which fails to be normally hyperbolic only at one point.

\begin{proposition} We can choose appropriate coordinates $(Z_2,\e)$ in $\Delta_2^{\en,-\y}$ such that the transition $\Pi_2^{-\y}: \Delta_2^{\en,-\y}\to\Delta_2^{\ex,-\y}$, restricted to $\r=0$, is an exponential type map of the form

\eq{
\Pi_2^{-\y}(0,\ve_2,Z_2)=\left( 0,\e,\phi_2(\e)+Z_2\exp\left( -\frac{1}{\e}(A_0+\e\psi_2(Z_2,\e)) \right)\right),
}

where $\phi_2$ are flat at $\e=0$, $\psi_2$ is $\Cl$-admissible, and where $A_0$ is given by the slow divergence integral of $X_{-\bar{b}}|_{\left\{ \r=0 \right\}}$.

\end{proposition}

\begin{proof} 

To prove that $A_0$ is given by the slow divergence integral we proceed along the same reasoning as in \cref{prop:slowdiv}, so we do not repeat it here. In figure \cref{fig:ky3} we see the three transitions that we must consider.

\begin{figure}[htbp]\centering
  \begin{tikzpicture}
    \pgftext{\includegraphics[scale=1.5]{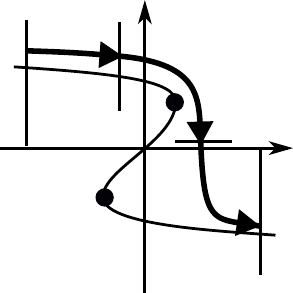}}
    \node at (0,2.5) {$z_2$};
    \node at (2.5,0) {$a_2$};
    \node at (-2.25,2.1) {$\Delta_2^{\en,-b_2}$};
    \node at (2.25,-2.1) {$\Delta_2^{\ex,-b_2}$};
    \node at (-.45,.35) {$\Omega^{\en}$};
    \node at (1.5,0.3) {$\Omega^{\ex}$};
    \end{tikzpicture}
  \caption{The three different transitions in which $\Pi_2^{-\y}$ is decomposed. The central transitions is locally an $A_2$ problem. The other two transitions at the sides are regular.}
  \label{fig:ky3}
\end{figure}
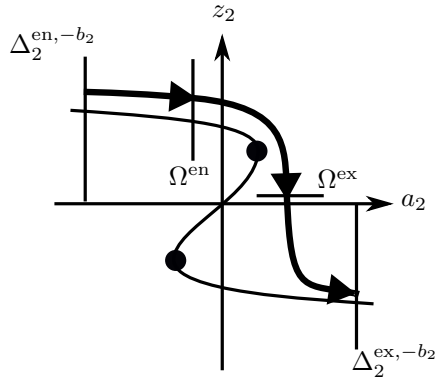

The three transitions are defined as

\eq{
\Pi_2^{reg_1} &:\Delta_2^{\en,-b_2}\to \Omega^{\en}\\
\Pi_2^{fold} &: \Omega^{\en}\to \Omega^{\ex}\\
\Pi_2^{reg_2} &:\Omega^{\ex}\to \Delta_2^{\ex,-b_2},
}

where we define $\Omega^{\en}$ and $\Omega^{\en}$ as

\eq{
  \Omega^{\en} &=\left\{ (\x,\e,Z_2)\in\R^3 \, | \, \x=-a_{2,\en}\right\}\\
  \Omega^{\ex} &=\left\{ (\x,\e,Z_2)\in\R^3 \, | \, Z_2=-Z_{2,\ex}\right\},
}

where $a_{2,\en}$ and $Z_{2,\ex}$ are sufficiently small positive constants. The total transition $\Pi_2^{+\y}$ is given by $\Pi_2^{\y}=\Pi_2^{reg_2}\circ\Pi_2^{fold}\circ\Pi_2^{reg_1}$. Recall from \cref{sec:Exp_trans} that if we want to write the transition  $\Pi_2^{+\y}$ as an exponential type map, we require that $\Pi_2^{reg_1}$ is expressed as an exponential type map with no shift. The transition $\Pi_2^{fold}$  is studied in e.g. \cite{JardonThesis,Krupa3}. In \cite{JardonThesis} is proved that there are local coordinates $(\bar Z_2,\ve)$ in $\Omega^{\en}$,  and $(\tilde a_2,\tilde\ve)$ in $\Omega^{\ex}$, such that the transition $\Pi_2^{fold}$ is given by

\eq{
\Pi_2^{fold}(\bar Z_2,\e)&=(\tilde a_2,\tilde\e)\\
&=\left(\e^{2/3}+O(\e), \e\right).
}

Assume now that we have characterized an invariant manifold $\M_{\e}^{fold}$ from $\Omega^{\en}$ to $\Omega^{\ex}$ via the map $\Pi_2^{fold}$. Now we want to ``extend'' $\M_{\e}^{fold}$ all the way up to the sections $\Delta_2^{\en,-b_2}$ and $\Delta_2^{\ex,-b_2}$ via transitions along normally hyperbolic regions of $\Gamma_2$. For this, it is more convenient to regard $\M_{\e}^{fold}$ as a graph $\zeta_2=\phi_{\e}(A_2)$ where $(\zeta_2,A_2)$ are local coordinates around the fold point $p_+$ and where $\phi_{\e}$ is a diffeomorphism for $\e>0$. In this way we can equivalently express the map $\Pi_2^{fold}$ as
\eq{
  \Pi_2^{fold}(\zeta,\e) &=(\tilde\zeta_2,\tilde\e)\\
  &=(\psi_{\e}(\zeta),\e)
}

where $\psi_{\e}$ is a diffeomorphism for $\e>0$ and only a homeomorphism for $\e=0$. Next, following \cref{sec:nf_reg} we can find coordinates $(Z_2,\e)$ in $\Delta_2^{\en,-b_2}$, and coordinates $(\tilde Z_2,\e)$ in $\Delta_2^{\ex,-b_2}$ in such a way that the transitions $\Pi_2^{reg_1}$ and $\Pi_2^{reg_2}$ are given as

\eq{
\Pi_2^{reg_1}(Z_2,\e) &=\left( Z_2\exp\left( -\frac{1}{\e}(A_0-a_{2,\en}) \right) \right) = (\bar Z_2,\e)\\
\Pi_2^{reg_2}(-Z_{2,\ex},\e) &=\left( -Z_{2,\ex}\exp\left( -\frac{1}{\e}(A_0-\tilde a_2) \right) \right) = (\tilde Z_2,\e).
}
  
\begin{remark}
  Recall that along normally hyperbolic slow manifolds, it is possible to make a normal form transformation in such a way that this transformation respects certain constraint or structure of the vector field, \cite{Bonckaert1,Bonckaert2}. In this particular case, we respect the choice of the invariant manifold $\M_{\e}^{fold}$.
\end{remark}

Next, we can compute the composition $\Pi_2^{-\y}=\Pi_2^{reg_2}\circ\Pi_2^{fold}\circ\Pi_2^{reg_1}$ by following \cref{sec:Exp_trans} and it thus follows that

\eq{
  \Pi_2^{-\y}(0,Z_2,\e)=\left(0,\bar\psi_{\e}+Z_2\exp\left( -\frac{1}{\e}(A_1+A_3+\e\psi_2) \right),\e \right),
}%
where $\bar\psi_{\e}=\psi_{\e}(0)\exp\left( -\frac{A_3}{\e} \right)$ and where $\psi_2=\psi_2(Z,\e)$ is a $\Cl$-admissible function. Note that $\bar\psi_{\e}$ is flat at $\e=0$.

\end{proof}

\subsection{Proof of \cref{teo:main}}\label{sec:proofofmain}
%\subfile{composition.tex}
\renewcommand{\be}{\bar{\varepsilon}}
\newcommand{\by}{\bar{b}}
\newcommand{\A}{\mathcal{A}}

Let us first recall that, within the blow up space, we have three types of transitions according to the initial condition $b_1|_{\Delta_1^{\en}}$, namely
\begin{itemize}
  \item If $b_1|_{\Delta_1^{\en}}\in O(\ve_1^{2/5})$ then we construct a transition passing through the charts $K_{\en}\to K_{\be}\to K_{\ex}$.
  \item If $b_1|_{\Delta_1^{\en}}\notin O(\ve_1^{2/5})$ and $b_1|_{\Delta_1^{\en}}>0$ then we construct a transition passing through the charts $K_{\en}\to K_{+\bar{b}}\to K_{\ex}$.
  \item If $b_1|_{\Delta_1^{\en}}\notin O(\ve_1^{2/5})$ and $b_1|_{\Delta_1^{\en}}<0$ then we construct a transition passing through the charts $K_{\en}\to K_{-\bar{b}}\to K_{\ex}$.
   
\end{itemize}
In \cref{fig:cusp_comp} we give a qualitative diagram of the local transitions obtained and their relationship. 

\tikzstyle{stuff_fill}=[rectangle,draw,preaction={fill=white}]

\begin{figure}[htbp]\centering
  \begin{tikzpicture}
    \pgftext{\includegraphics[scale=0.75]{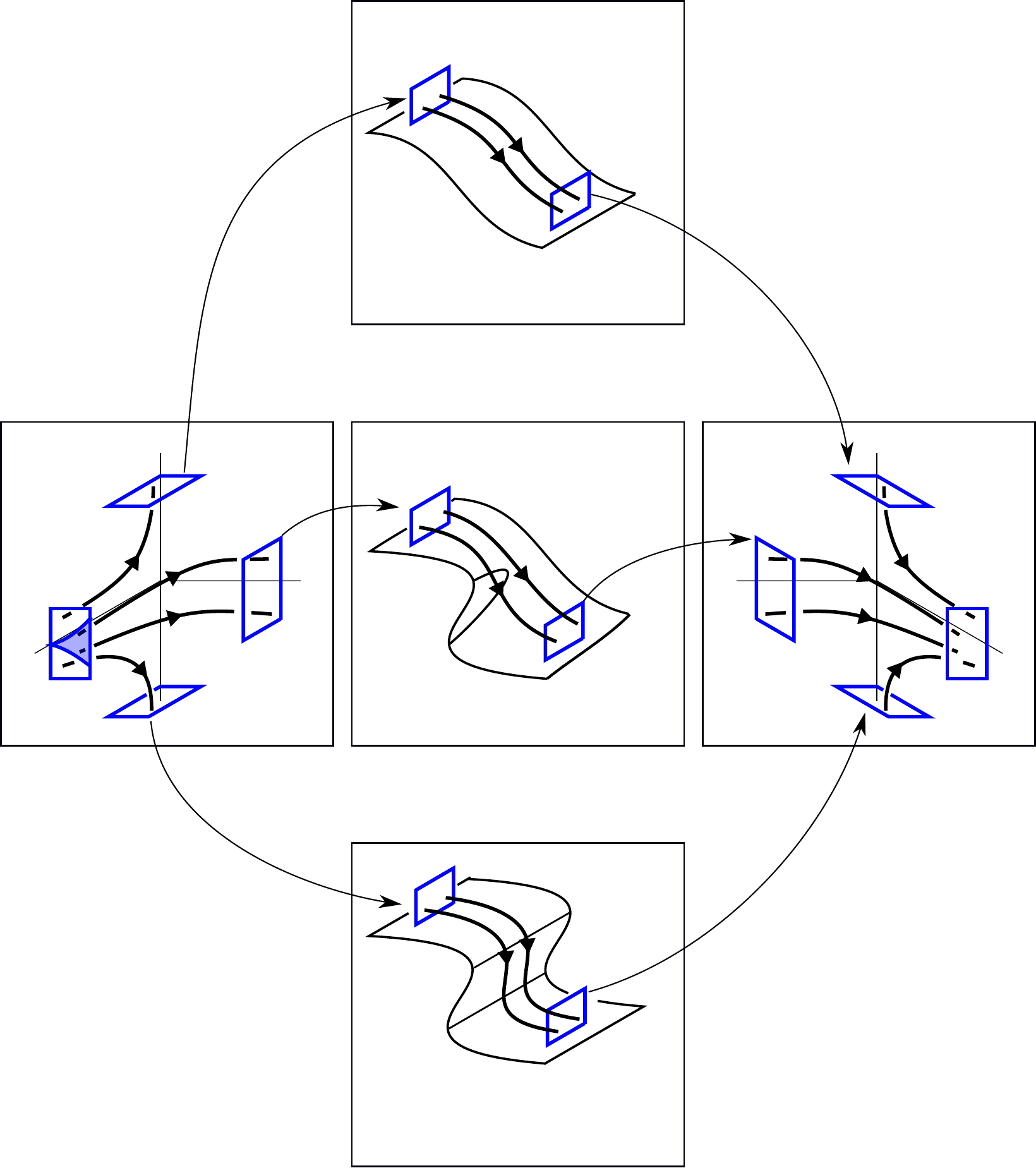}}
    \node at (-5,-2.3) {$K_{\en}$};
    \node at (0,-2.3) {$K_{\be}$};
    \node at (5,-2.3) {$K_{\ex}$};
    \node at (0,2.7) {$K_{+\by}$};
    \node at (0,-7.6) {$K_{-\by}$};

    \node at (-6.05,-.8) {\scriptsize{$r_1$}};
    \node at (-4.3,1.75) {\scriptsize{$b_1$}};
    \node at (-2.4,0) {\scriptsize{$\ve_1$}};

    \node at (-6.05+8.475,.05) {\scriptsize{$\ve_3$}};
    \node at (-4.3+8.65,1.75) {\scriptsize{$b_3$}};
    \node at (-2.4+8.45,-.9) {\scriptsize{$r_3$}};

    \node at (-5.6,-1.45) {\scriptsize{$\Delta_1^{\en}$}};
    \node at (-3.7,-1.65) {\scriptsize{$\Delta_1^{\ex,-b_1}$}};
    \node at (-3.65, 0.95) {\scriptsize{$\Delta_1^{\ex,+b_1}$}};
    \node at (-2.75, -.8) {\scriptsize{$\Delta_1^{\ex,\ve_1}$}};

    \node at (3,-.85) {\scriptsize{$\Delta_3^{\en,\ve_3}$}};
    \node at (5.6,-1.65) {\scriptsize{$\Delta_3^{\en,-b_3}$}};
    \node at (5.6, 0.95) {\scriptsize{$\Delta_3^{\en,+b_3}$}};
    \node at (5.8, -.1) {\scriptsize{$\Delta_3^{\ex}$}};

    \node at (-2.25,.85) [stuff_fill] {\scriptsize{$M_{\en}^{\be}$}};
    \node at (-3.65,3.85) [stuff_fill] {\scriptsize{$M_{\en}^{+\bar{b}}$}};
    \node at (3.45,2.95) [stuff_fill] {\scriptsize{$M_{+\bar{b}}^{\ex}$}};
    \node at (1.5,0.25) [stuff_fill] {\scriptsize{$M_{\bar\ve}^{\ex}$}};
    \node at (2.5,-3.95) [stuff_fill] {\scriptsize{$M_{-\bar{b}}^{\ex}$}};
    \node at (-3.5,-2.95) [stuff_fill] {\scriptsize{$M_{\en}^{-\bar{b}}$}};

    \node at (-.6,1.3) {\scriptsize{$\Delta_2^{\en,\ve_2}$}};
    \node at (1,-1.3) {\scriptsize{$\Delta_2^{\ex,\ve_2}$}};

    \end{tikzpicture}
  \caption{All the transitions obtained in the charts. We have to compose all such transitions through the matching maps $M_i^j$. A matching map $M_i^j$ relates the coordinates between the charts $K_i$ and $K_j$. }
  \label{fig:cusp_comp}
\end{figure}
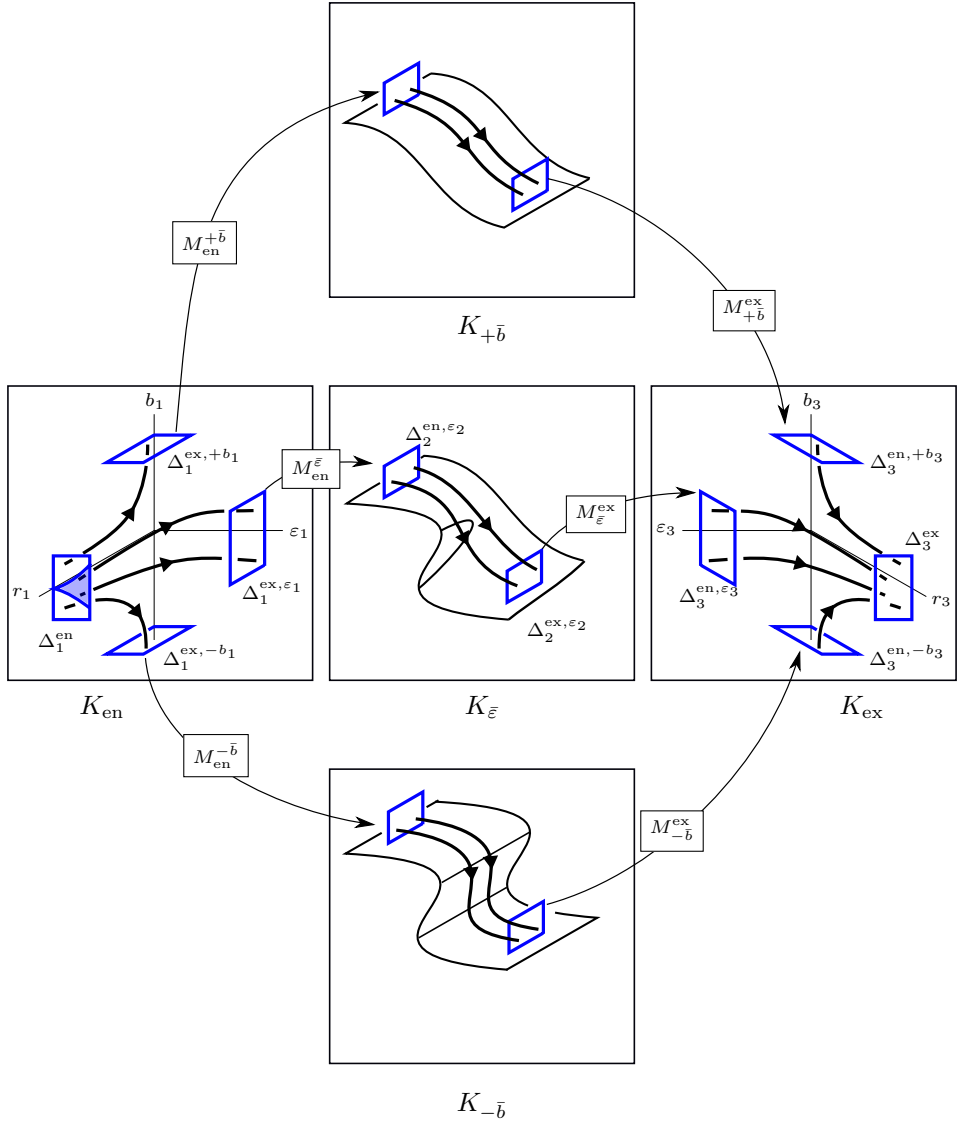

Let us only detail the transition through the inner layer $\Delta^{\inn}$ corresponding to $b_1|_{\Delta_1^{\en}}\in O(\ve_1^{2/5})$, the other cases follow the same lines. 

The transition $\Pi^{\inn}:\Delta_1^{\inn}\to\Delta_2^{\ex}$ is given as

\renewcommand{\x}{a}
\renewcommand{\y}{b}

\eq{
  \Pi^{\inn} &= \Pi_{3}^{\ve_3}\circ M_{\bar\ve}^{\ex}\circ\Pi_{2}^{\ve_2}\circ M_{\en}^{\bar\ve}\circ\Pi_1^{\inn}
}

where the matching maps are obtained from the blow-up map. For example, to obtain the matching map from the chart $K_{\en}$ to the chart $K_{\be}$ we relate the two directional blow-up maps

\eq{
  \x=-r_1^3, \; \y = r_1^2 b_1, \; z= r_1z_1, \; \ve=r_1^5\ve_1
}

and 

\eq{
  \x=r_2^3 a_2, \; \y = r_2^2 b_2, \; z= r_2z_2, \; \ve=r_2^5.
}

Let us work out only with the $z$-component of the transitions as it is the only relevant one. Recall from \cref{Ken} that $\Pi_1^{\inn}$ is an exponential type map with no shift. Next, the composition $\Pi^{central}=M_{\bar\ve}^{\ex}\circ\Pi_{2}^{\ve_2}\circ M_{\en}^{\bar\ve}$ yields a diffeomorphism as $\Pi_2^{\ve_2}$ is a diffeomorphism, and the matching maps are also diffeomorphisms on their domain of definition. Next, the last transition $\Pi_3^{\ve_3}$ is an exponential type map with no shift, see \cref{Kex}. Therefore, following \cref{sec:Exp_trans} we have that $\Pi_3^{\ve_3}\circ\Pi^{central}\circ\Pi_1^{\inn}$ is an exponential type map of the form

\eq{
  \Pi^{\inn}_{Z_1}=\bar\phi(B_1,\ve_1)+Z_1\exp\left( -\frac{1}{\ve_1}\left( \bar \A(B_1,\ve_1)+\ve_1\bar\Psi(B_1,\ve_1,Z_1) \right) \right),
}

where $\bar \A>0$ and $\phi$ and $\Psi$ are $\Cl$ admissible functions. The differentiability of $\phi$ and $\Psi$ with respect to monomials is evident from the results of \cref{Ken}. By blowing down we obtain that the transition $\Pi^{\inn}:\Sigma^{\en}\to\Sigma^{\ex}$ (in a small neighborhood of the cusp point and within the inner layer as domain) reads as

\eq{
  \Pi^{inner}_{Z}=\phi(B,\ve)+Z\exp\left( -\frac{1}{\ve}\left(  \A(B,\ve)+\ve\bar\Psi(B,\ve,Z) \right) \right).
}

To obtain the transition $\Pi:\Sigma^-\to\Sigma^+$ we now need to compose $\Pi^{\inn}_{Z}$ with exponential type maps on the left and on the right corresponding to

\eq{
  \Pi^-:\Sigma^-\to\Sigma^{\en}\\
  \Pi^+:\Sigma^{\ex}\to\Sigma^+.
}

However, we must proceed with care. In order to express the transition $\Pi$ as an exponential type map, we need to choose appropriate coordinates on $\Sigma^{-}$ and on $\Sigma^+$ that respect the already chosen coordinates in $\Sigma^{\en}$ and in $\Sigma^{\ex}$. Fortunately, this is possible with the extensions of Bonckaert \cite{Bonckaert1,Bonckaert2} to the normalization results of Takens \cite{Takens_partially}.

For sake of clarity, let $(B_{\en},Z_{\en})$ be coordinates in $\Sigma^{\en}$ and $(B_{\ex}, Z_{\ex})$ be coordinates in $\Sigma^{\ex}$. We have shown that these coordinates can be chosen in such a way that the ``vertical'' component of the transition map $\Pi^{\inn}:\Sigma^{\en}\to\Sigma^{\ex}$ reads as

\eq{
  \Pi_{Z_{\en}}(B_{\en},Z_{\en},\ve)&=Z_{\ex}\\
                            &=\phi(B_{\en},\ve)+Z_{\en}\exp\left( -\frac{1}{\ve}\left(  \A(B_{\en},\ve)+\ve\bar\Psi(B_{\en},\ve,Z_{\en}) \right) \right).
}

In this case the invariant manifold, say $\M_{\ve}$, is given by $Z_{\en}=0$. Using \cite{Bonckaert1,Bonckaert2} we can find suitable coordinates $(B_-,Z_-)$ in $\Sigma^-$ in such a way that

\eq{
  \Pi^-_{Z_-}(B_-,Z_-,\ve)=Z_-\exp\left( -\frac{1}{\ve}(A_0) \right)=Z_{\en}.
}

In other words, there is a change of coordinates respecting the invariant manifold $\M_{\ve}$ under which the transition $\Pi^-$ is an exponential type map with no shift and linear. Similar arguments hold for the choice of coordinates in $\Sigma^+$. Finally, following \cref{sec:Exp_trans}, the composition $\Pi^+_{Z_+}\circ\Pi_{Z_{\en}}\circ\Pi^-_{Z_-}$ leads to the result.

\appendix
\renewcommand*{\thesection}{\Alph{section}}

\section{Exponential type functions}\label{sec:Exp_trans}
%\subfile{ExpTrans.tex}
In this section, we discuss a particular type of function which will be found and used frequently throughout the main text. First, however, let us give two preliminary definitions. 
\begin{definition}[ $\Cl$-admissible function] Let $U\in\R^n$.  A function $f:\R^n\to\R$ is said to be a \emph{$\Cl$-admissible function} if $f$ is $\Cl$-smooth away form the origin (for any $\ell>0$), $\mathcal C^0$ at the origin and if for all $n_i\in \N$ and $n_i<\ell$, there exists an $N(n_i)\in\N$ such that
\eq{
\parcs{^{n_i} f}{U_i^{n_i}}\in O\left( U_i^{-N(n_i)} \right), \qquad \text{ as } U_i\to 0.
}
  
\end{definition}

Now, we define a particular type of differentiability. For this we need to extend the common concept of monomial. In our context, a monomial, e.g. in two variables, $\omega(u,v)$ is any expression of the form $u^\alpha v^\beta$ or of the form $u^\alpha(\ln v)^\beta$, with $\alpha,\beta\in\R$.  In general, if we let $u\in\R^m_{}$ and $v\in\R^n$, we allow a monomial $\omega$ to be any expression of the type $u^p(\ln v)^{q}$, where $u^p=u_1^{p_1}\cdots u_m^{p_m}$ and $(\ln v)^{q}=(\ln v_1)^{q_1}\cdots(\ln v_n)^{q_n}$. We note that these monomials are admissible functions.

\begin{definition}[$\Cl$-function with respect to monomials]\label{def:mon} Let $(U,V)\in\R^m\times \R^n$. We say that a function $f(U,V)$ is a $\Cl$-function with respect to a monomial $\omega$, if $f$ is $\Cl$ w.r.t. $V$ in a neighborhood of $0\in\R^n$, and if there is a quadrant $\mathcal U=[0,u_1)\times\cdots\times[0,u_n)\subset\R^m$ where the monomial $\omega$ is defined and such that the function $\tilde f(\omega,U,V)=f(U,V)$ is $\Cl$ with respect to $\omega$ in $\mathcal U$. Similarly, the function $f$ is said to be a $\Cl$-function with respect to the monomials $\omega_1,\ldots,\omega_s$ if there is a  quadrant $\mathcal U$ where the monomials are defined and such that the function $\tilde f(\omega_1,\ldots,\omega_s,U,V)=f(U,V)$ is $\Cl$ with respect to $\omega_1,\ldots,\omega_s$ in $\mathcal U$.
\end{definition}

Observe that a function $f$ which is differentiable w.r.t monomials is an admissible function. As an example, consider $f(U)=U_1\ln U_1\phi(U)$ where $\phi(U)$ is smooth. This function is smooth away from $U=0$ and $C^0$ at the origin. However, it is not differentiable w.r.t. $U_1$ at $U_1=0$ but it is differentiable with respect to $\omega=U_1\ln U_1$ at $\omega=0$.\bigskip

Let $V\in\R^m$, $Z\in\R$, and as usual $\ve$ denotes a small parameter. 

\renewcommand{\A}{\mathcal{A}}
\newcommand{\B}{\mathcal{B}}

  \begin{definition}[Exponential type function]\label{def:exptype} A function $D(V,Z,\ve)$ is called of exponential type if it has the following form

  \eq{\label{eq:defexp}
  D(V,Z,\ve)=\B(V,\ve)+Z\exp\left( -\frac{\A(V,\ve)+\Phi(V,\ve,Z)}{\ve} \right),
  }

  where $\A$ and $\B$, are $\Cl$ admissible functions with $\A>0$, and $\B(V,0)=0$; and where $\Phi$ is $\Cl$ in $z$ and $\Cl$ w.r.t. monomials of $(V,\ve)$ with $\Phi(V,0,Z)=0$. We distinguish two particular cases

  \begin{enumerate}
    \item The exponential type function $D$ is \emph{without shift} if $\B\equiv 0$.
    \item The exponential type function $D$ is \emph{linear} if $\Phi(V,Z,\ve)\equiv\Phi(V,\ve)$.  
  \end{enumerate}

  \end{definition}

\begin{remark} Given a function $D$ and if it is of exponential type, the representation of $D$ is unique in the sense that all the functions in r.h.s of \cref{eq:defexp} are computable from $D$. In fact

\eq{
  \B&=D(V,0,\ve)\\
  \A&=\lim_{Z\to 0}\left( -\ve\ln\left(\frac{D(V,Z,\ve)-D(V,0,\ve)}{Z}\right) \right)\\
  \Phi&=-\ve\ln\left(\frac{D(V,Z,\ve)-D(V,0,\ve)}{Z}\right)-\A.
}

\end{remark}

We want to study the scenario where we have to compose $D$ with some other functions and want to keep the exponential type structure. To be more precise, we consider $D$ as an $(V,\ve)-$parameter family of functions (in $Z$) and compose it with a $(V,\ve)-$parameter family of diffeomorphisms $\Psi_{(V,\ve)}$ on $\R$.

\begin{proposition}[Composition on the left]\label{prop:compleft} Let $\Psi_{(V,\ve)}:\R\to\R$ be a family of diffeomorphisms, and let $D$ be an exponential type function. Then, the composition $\Psi_{(V,\ve)}\circ D$ is also of exponential function of the form
\eq{
  \tilde D=\tilde \B(V,\ve)+Z\exp\left( -\frac{\A(V,\ve)+\tilde\Phi(V,Z,\ve)}{\ve} \right),
}
where $\tilde\B$ and $\tilde\Phi$ are admissible functions. 
\end{proposition}
\begin{proof}
  Let us simplify the notation by writing $\Psi=\Psi_{(V,\ve)}$. Since $\Psi$ is a diffeomorphism we can write $\Psi(a+b)=\Psi(a)+C(1+\psi(a,b))b$, near $b=0$, with $\psi$ a $\Cl$ function such that $\psi(a,0)=0$ and with $C>0$. Then we have
  \eq{
  \Psi\circ D(z)  &=\Psi\left( \B+Z\exp\left( -\frac{\A+\Phi}{\ve} \right)  \right)\\
          &=\Psi(\B)+C(1+\psi(V,Z,\ve))Z\exp\left( -\frac{\A+\Phi}{\ve} \right).
  }

  Since $C>0$ we can take the logarithm of $C(1+\psi(V,Z,\ve))$ and then we have
  \eq{
  \Psi\circ D(z)  &= \Psi(\B)+\exp(\ln(C(1+\psi))Z\exp\left( -\frac{  \A+\Phi}{\ve} \right)\\
          &= \Psi(\B)+Z\exp\left( -\frac{ \A+\Phi+\ve\ln(C(1+\psi)}{\ve} \right).
  }

  The result is obtained by setting $\tilde \B=\Psi(\B)$ and $\tilde\Phi=\Phi+\ve\ln(C(1+\psi)$.
\end{proof}

\begin{proposition}[Composition on the right]\label{prop:compright} Let $\Psi_{(V,\ve)}:\R\to\R$ be a family of diffeomorphisms with no shift, that is $\Psi_{(V,\ve)}(0)=0$ for all $(V,\ve)$, and let $D$ be an exponential type function. Then, the composition $D\circ\Psi_{(V,\ve)}$ is also of exponential function of the form
\eq{
  \tilde D=\tilde \B(V,\ve)+Z\exp\left( -\frac{\A(V,\ve)+\tilde\Phi(V,Z,\ve)}{\ve} \right),
}
where $\tilde\B$ and $\tilde\Phi$ are admissible functions. 
\end{proposition}

\begin{proof}
  Let us simplify the notation by writing $\Psi=\Psi_{(V,\ve)}$. Since $\Psi(0)=0$ we can write $\Psi(z)=C(1+O(z))z$ with $C>0$. Then we have

  \eq{
    D\circ\Psi(z) &=D(C(1+O(z))z)=\B(V,\ve)+C(1+O(z))z\exp\left( -\frac{\A(V,\ve)+\Phi(V,\ve,\Psi)}{\ve} \right)\\
            &= \B(V,\ve)+z\exp\left( -\frac{\A(V,\ve)+\Phi(V,\ve,\Psi)+\ve\ln(C(1+O(z)))}{\ve} \right).
  }

  The result then is obtained by setting $\tilde\Phi=\Phi(V,\ve,\Psi)+\ve\ln(C(1+O(z)))$.
\end{proof}

\begin{remark} If we want the composition $\Pi\circ\Psi_{(V,\ve)}$ to be of exponential type, the family $\Psi_{(V,\ve)}$ cannot be arbitrary. In order to preserve the ``exponential structure'', $\Psi_{(V,\ve)}$ should satisfy the hypothesis of \cref{prop:compright}. In \cref{coro2} we show a particular case in which the diffeomorphism $\Psi$ can have a shift and yet preserve the structure of the exponential type function.
\end{remark}

Let us proceed by presenting a couple of useful corollaries.

\begin{corollary}\label{coro1} Let $D_1$ and $D_2$ be two exponential type functions of the form
\eq{
D_1(V,Z,\ve) &= Z\exp\left( -\frac{\A_1(V,\ve)+\Phi_1(V,Z,\ve)}{\ve} \right)\\
D_2(V,Z,\ve) &= \B_2(V,\ve)+Z\exp\left( -\frac{\A_2(V,\ve)+\Phi_2(V,Z,\ve)}{\ve} \right),
}

that is, $D_1$ is an exponential type function with no shift. Then $D_2\circ D_1$ is an exponential type function.
  
\end{corollary}
% \begin{proof}
%   The proof is performed by setting $D_1=\Psi$ in \cref{prop:compright}. The result of the composition reads as

%   \eq{
%   \tilde D=D_2\circ D_1(z)=\B_2(\ve)+Z\exp\left( -\frac{\tilde \A(V,\ve)+\tilde\Phi(V,\ve,Z)}{\ve}\right),
%   }

% where $\tilde \A=\A_1+\A_2$ and $\tilde\Phi(V,\ve,Z)=\Phi_1(V,\ve,Z)+\Phi_2(V,\ve,D_1)$.

% \end{proof}

\begin{corollary}\label{coro2} Let $D_1$ and $D_2$ be two exponential type functions with $D_2$ linear, this is
\eq{
D_1(V,Z,\ve) &= \B_1(V,\ve)+Z\exp\left( -\frac{\A_1(V,\ve)+\Phi_1(V,Z,\ve)}{\ve} \right)\\
D_2(V,Z,\ve) &= \B_2(V,\ve)+Z\exp\left( -\frac{\A_2(V,\ve)}{\ve} \right).
}

Then the composition $D_2\circ D_1$ is of exponential type.
\end{corollary}

% \begin{proof} The proof is performed by setting $D_2=\Psi$ in \cref{prop:compleft}. It is important to note that in the composition $\tilde D=D_2\circ D_1$ we get

%   \eq{
%   \tilde D=\tilde \B+Z\exp\left( -\frac{\tilde \A+\tilde\Phi}{\ve}\right),
%   }
  
%   where $\tilde \B=\B_2+\B_1\exp\left( -\frac{\A_2}{\ve} \right)$ and where $\tilde \A=\A_1+\A_2$. This means that the shift induced by $D_1$ is flat.
% \end{proof}

It is useful to consider the following: let $X(V,Z,\ve)$ be a given vector field on $\R^{m+2}$, and let $\Sigma_0$ and $\Sigma_1$ be codimension $1$ subsets of $\R^{m+2}$ which are transversal to the flow of $X$. For the moment it is sufficient to think of a section $\Sigma_i$ given by $\left\{ V_j=v_0 \right\}$ or by $\left\{ \ve=\ve_0 \right\}$ with $v_0$ and $\ve_0$ fixed constants. Induced from \cref{def:exptype} we then have the following.

\begin{definition}[Exponential type transition] A transition $\Pi:\Sigma_0\to\Sigma_1$ is called of exponential type if and only if its $Z$-component is an exponential type function. This is, an exponential type transition is of the form
\eq{
  \Pi(V,Z,\ve)&= (G,D,H)\\
  &=\left( G(V,\ve), \, \B(V,\ve)+Z\exp\left( -\frac{\A(V,\ve)+\Phi(V,Z,\ve)}{\ve}\right),\,  H(V,\ve) \right) ,
}

where $G:\R^{m+1}\to\R^m$ and $H:\R^{m+1}\to\R$ are $\Cl$ with $G(V,0)=V$ and $H(V,0)=0$; where  $A$, $B$ and $\Phi$ are $\Cl$-admissible functions. The names exponential type transition with no shift and linear are inherited as well from the type of $D$.
  
\end{definition}

Suppose now that $X$ is a given vector field on $\R^{m+2}$, as above, and let $\Sigma_i$ with $i=0,1,2,3,4,5$, be disjoint sections which are all transversal to the flow of $X$. Assume that $X$ induces exponential type transitions $\Pi_i:\Sigma_{i-1}\to\Sigma_i$ with $i=1,2,3,4,5$ of the following form

\begin{enumerate}
  \item $\Pi_1$ is with no shift and linear
  \item $\Pi_2$ is with no shift 
  \item $\Pi_3$ is a general diffeomorphism 
  \item $\Pi_4$ is with no shift 
  \item $\Pi_5$ is with no shift and linear. 
\end{enumerate}

We need to show that the composition of all these five maps is an exponential type transition.

\begin{proposition}\label{prop:expcomp} Let $\Pi_i:\Sigma_{i-1}\to\Sigma_i$ as described above. Then the composition $\Pi=\Pi_5\circ\Pi_4\circ\Pi_3\circ\Pi_2\circ\Pi_1$ is an exponential type map of the form
  \eq{ 
    \Pi=\left(\tilde G(V,\ve),\, \tilde \B(V,\ve)+Z\exp\left( -\frac{\tilde \A(V,\ve)+\tilde\Phi(V,Z,\ve)}{\ve} \right), \, \tilde H(V,\ve) \right),
  }

  where $\tilde \A=\A_1+\A_2+\A_4+\A_5$.
\end{proposition}

\begin{proof}
  Let us write each of the transitions as follows.

  \begin{enumerate}
    \item $\Pi_1(V,Z,\ve)=\left( G_1,D_1,H_1 \right)=\left( G_1,Z\exp\left( -\frac{\A_1(V,\ve)}{\ve} \right), H_1 \right)$ 
    \item $\Pi_2(V,Z,\ve)=\left( G_2,D_2,H_2 \right)=\left( G_2,Z\exp\left( -\frac{\A_2(V,\ve)+\Phi_2(V,Z,\ve)}{\ve} \right), H_2 \right)$ 
    \item $\Pi_3(V,Z,\ve)=\left( G_3,D_3,H_3 \right)$
    \item $\Pi_4(V,Z,\ve)=\left( G_4,D_4,H_4 \right)=\left( G_4,Z\exp\left( -\frac{\A_4(V,\ve)+\Phi_4(V,Z,\ve)}{\ve} \right), H_4 \right)$
    \item $\Pi_5(V,Z,\ve)=\left( G_5,D_5,H_5 \right)=\left( G_5,Z\exp\left( -\frac{\A_5(V,\ve)}{\ve} \right), H_5 \right)$ 
  \end{enumerate}

  For brevity let $\Pi_2\circ\Pi_1=(\tilde G_2,\tilde D_2,\tilde H_2)$. Then we have
  \eq{
  &(\tilde G_2,\tilde D_2,\tilde H_2) =\\
  &  \left( G_2(G_1,H_1),D_1\exp\left( -\frac{\A_2(G_1,H_1)+\Phi_2(G_1,D_1,H_1)}{H_1} \right), H_2(G_1,H_1)  \right).
  }

  Now, we take care only of the $Z$-component of the composition $\Pi_2\circ\Pi_1$. From the hypothesis on $G_1$ and $H_1$ we can write $G_1=V+O(\ve)$ and $H_1=\alpha\ve(1+O(\ve))$ with $\alpha>0$, then 
  \eq{
    \tilde D_2= Z\exp\left( -\frac{\A_1(V,\ve)+\A_2(V,\ve)+\bar\Phi_2(V,Z,\ve)}{\ve} \right),
  }

  where we have gathered in $\bar\Phi_2$ the function $\Phi_1$ and the terms resulting from taking $G_1=V+O(\ve)$ and $H_1=\alpha\ve(1+O(\ve))$. In a similar way, letting $\Pi_5\circ\Pi_4=(\tilde G_5,\tilde D_5,\tilde H_5 )$ we get
  \eq{
  \tilde D_5=Z\exp\left( -\frac{\A_4(\ve)+\A_5(\ve)+\bar\Phi_5(V,Z,\ve)}{\ve} \right)
  }

  Next, and following similar arguments as above, we know from \cref{prop:compleft} that the composition $\Pi_{321}=\Pi_3\circ\Pi_2\circ\Pi_1$ is of exponential type \emph{with shift}. Finally since the transition $\Pi_{54}=\Pi_5\circ\Pi_4$ is of exponential type with no shift, and using \cref{prop:compleft}, we have that $\Pi_{54}\circ\Pi_{321}$ is an exponential type transition as claimed in the proposition.

  \begin{remark}
    In the case where $\Pi_3$ is an exponential type map, we get a similar result with $\tilde \A=\A_1+\A_2+\A_3+\A_4+\A_5$.
  \end{remark}
\end{proof}

\section{First order differential equations (by R. Roussarie)}
%\subfile{appRobert.tex}
The contents of this section shall appear in greater detail in \cite{DumRouMaelBook}. We reproduce some results here for completeness purposes and to use them in \cref{sec:nfs}.\medskip

Let $X(x)$ be a smooth vector field defined on $W\subset \R^n$, for arbitrary $n\in\N$ (here we include the possible parameters). Let $G(x,y):W\times\R\to\R$ be a smooth function. We shall study the solutions of the first order differential equation
\eq{\label{ro1}
  X\cdot K(x)=G(x,K(x)),
}
where $K(x)$ is the unknown function. We assume the following
\begin{enumerate}
  \item There exists an open section $\Sigma\subset W$ which is transverse to $X$.
  \item Let $\phi(t,x)$ denote the flow of $X$. We can choose an open domain $W_{\Sigma}$ with the property that for any $x\in W_{\Sigma}$, there exists a unique smooth time $t(x)$ (possibly unbounded) such that $\phi(t(x),x)\in\Sigma$. 
  \item The vector field $Z(x,y)=X(x)+G(x,y)\partial_y$ has a complete flow.
\end{enumerate}

The flow of $Z$ takes the form $(\phi(t,x),\psi(t,x,y))$, where $\phi$ is the flow of $X$. It follows that \emph{$K(x)$ is a solution of \cref{ro1} if and only if the graph $\left\{ y=K(x)\right\}$ is a surface tangent to the vector field $Z$}. Then we have the implicit formula
\eq{\label{ro2}
  \psi(t(x),x,K(x))=0.
}

In our applications, the function $G$ is affine in $y$, that is $G(x,y)=L(x)y+\Pi(x)$ where $L$ and $\Pi$ are smooth. If we write $\bar L(t,x)=L(\phi(t,x))$ and $\bar\Pi(t,x)=\Pi(\phi(t,x))$ (where $\phi$ is the flow of $X$), we have for $\psi$ the following linear differential equation
\eq{\label{ro3}
  \frac{d\psi}{dt}(t,x,y)=\bar L(t,x)\psi(t,x,y)+\bar\Pi(t,x).
}

Then we can integrate \cref{ro3} with the initial condition $\psi(0,x,y)=y$ to obtain
\eq{
  \psi(t,x,y)=\exp\left( \int_0^t\bar L(\tau,x)d\tau \right)\left\{ y+\int_0^t\bar\Pi(\tau,x)\left[ \exp\left( -\int_0^\tau\bar L(\sigma,x)d\sigma \right) \right] d\tau \right\}.
}

Since $\exp\left( \int_0^t\bar L(\tau,x)d\tau \right)>0$ we can solve the implicit equation \cref{ro2} obtaining
\eq{\label{ro4}
  K(x)=-\int_0^{t(x)}\Pi(\phi(\tau,x))\left[ \exp\left( -\int_0^\tau L(\phi(\sigma,x))d\sigma \right) \right] d\tau,
}
where we recall that $\phi$ is the flow of $X$ and $t(x)$ is the time to go from $x$ to the section $\Sigma$ along this flow.\medskip

Let us now assume that the vector field $X$ is partially hyperbolically attracting in the following sense: we assume coordinates $x=(a,b)\in\R^p\times\R^q$ and that the vector field $X$ has a decomposition $X(x)=U(x)+V(x)$ where $U$ is the component along $\R^p$ and $V$ is the component along $\R^q$. Moreover, we assume that $V=0$ on $\R^p\times\left\{ 0 \right\}$ (that is $X$ is tangent to $\R^p\times\left\{ 0 \right\}$). We also assume that at each point $x=(a,b)$ it is satisfied that $D_bV(a,0)$ has all its eigenvalues with strictly negative real part. We further suppose that $X$ is given on $W=D\times\Delta$ where $D$ is a domain diffeomorphic to a ball in $\R^p$ and $\Delta$ is a ball in $\R^q$. We choose $\Delta=\Delta_{\rho_0}$ for some $\rho_0>0$ where $\Delta_\rho=\left\{ b\in\R^q\,|\,||b||<\rho \right\}$. It then follows that under a linear change of coordinates $(a,b)\mapsto(a,A(a)b)$, the vector field $X$ enters along $D\times\partial\Delta_\rho$ for $0<\rho\leq\rho_0$ if we choose $\rho_0$ small enough. We now have the following

\begin{proposition}\label{propRob}
  Assume that $D_bV(a,0)$ has all its eigenvalues with a strictly negative real part and that $\rho_0$ is small enough as explained above. Let $B$ be any domain diffeomorphic to a closed ball inside the interior of $D$ and assume that the function $\Pi(x)$ is flat along $D\times\left\{ 0 \right\}$. Then the equation 
  \eq{\label{roprop}
  X \cdot K(x)=L(x)K(x)+\Pi(x)
  }
  has a smooth solution $K(x)$ in $B\times\Delta$ which is flat along $B\times\left\{ 0 \right\}$.
 \end{proposition} 

\begin{proof} Let $f(a):\R^p\to[0,1]$ be a smooth function which is equal to $1$ on $B$ and equal to $0$ on a neighborhood of $\partial D$. Define the vector field
\eq{
T=V+fU.
}
This vector field $T$ coincides with $X$ on $B\times\Delta$. Moreover, $T$ is tangent along $\partial D\times\Delta$ and enters the domain $D\times\Delta$ along $D\times\partial\Delta$. Let $\phi(t,x)=(\phi_a(t,x),\phi_b(t,x))\in\R^p\times\R^q$ denote the flow of $T$. It follows that $\phi(t,x)\in D\times\Delta$ for all $x\in D\times\Delta$ and all $t\geq 0$. From the assumption on $V$ we have that there exists a positive constant $E>0$ such that
\eq{\label{ro6}
||\phi_b(t,x)||\leq ||b||\exp(-Et),
}
for any $x=(a,b)\in D\times\Delta$ and $t\in [0,+\infty)$. We now want to use this flow $\phi$ in \cref{ro4} noting that if the integral converges, then $K(x)$ is a solution to the equation $T\cdot K=LK+\Pi$ on $D\times\Delta$ and then to the equation $X\cdot K=Lk+\Pi$ on $B\times\Delta$. In this setting \cref{ro4} is written as
\eq{\label{ro5}
K(x)=-\int_0^{\infty}\Pi(\phi(\tau,x))\left[ \exp\left( -\int_0^\tau L(\phi(\sigma,x))d\sigma \right) \right] d\tau.
}

Now, we need to prove that \cref{ro5} defines a smooth function on $D\times\Delta$ which is flat along $D\times\left\{ 0 \right\}$. In other words, we shall prove that $K$ and all its partial derivatives are equal to $0$ on $D\times\left\{ 0 \right\}$. As $L$ is bounded, there exists a constant $M_0>0$ such that
 \eq{
\exp\left( -\int_0^\tau L(\phi(\sigma,x))d\sigma \right) \leq\exp(M_0\tau).
 }  
 Next, let $N\in\N$. Since $\Pi$ is flat in $v$, there exists a constant $P_N>0$ such that
 \eq{
 |\Pi(a,b)|\leq P_N||b||^N,
 }
 and then from \cref{ro6} it follows that
 \eq{
 |\Pi(\phi(\tau,x))|\leq P_N||b||^N\exp(-NE\tau).
 }

 Using these estimates we have that

 \eq{\label{ro7}
 |K(x)|\leq P_N||b||^N\int_0^{+\infty}\exp((M_0-NE)\tau)d\tau.
 }

The integral in \cref{ro7} converges if $N$ is large enough, strictly speaking if $N>\tfrac{M_0}{E}$. This proves that by choosing $N$ sufficiently large, the right hand side of \cref{ro5} defines a function which is continuous and equal to $0$ on $D\times \left\{ 0 \right\}$.

Let us now consider any partial derivation $\partial_\alpha K$ of $K$. Let us write
\eq{
  H(\tau,x)=\Pi(\phi(\tau,x))\exp\left[ -\int_0^\tau L(\phi(\sigma,x) )d\sigma \right],
}

the integrand in \cref{ro5}. Using chain rule on the derivative of \cref{ro5}, we have to prove that the integral 
\eq{\label{ro8}
  \int_0^{+\infty} \partial_{\alpha} H(\tau,x)d\tau
}
is convergent and that there is an estimate similar to \cref{ro7} for $N$ large enough. We do not want to give all the details here and refer the reader to \cite{DumRouMaelBook}. The idea is that $\partial_{\alpha} H(\tau,x)$ is a finite sum of terms such that each of these terms is a product of factors which are partial derivatives in $x$ and are of of one of the following forms
\begin{enumerate}[leftmargin=*]
  \item $\partial_{\alpha_1}(\phi(\tau,x))$. Since $\Pi$ is smooth and flat along  $D\times \left\{ 0 \right\}$, this is also the case for $\partial_{\alpha_1}(\phi(\tau,x))$. Therefore, for $N$ sufficiently large, we can write an estimate of the form
  \eq{
  | \partial_{\alpha_1}(\phi(\tau,x))|\leq P_{N_{\alpha_1}}||b||^N\exp(-NE\tau),
  }
  for constants $P_{N_{\alpha_1}}>0$.
  \item $\partial_{\alpha_2}\phi(\tau,x)$ (resp. $\partial_{\alpha_2}\phi(\sigma,x)$, note that $0\leq \sigma\leq\tau$). By the usual variational method along trajectories, there exists constants $E_{\alpha_2}>0$ such that $|\partial_{\alpha_2}\phi(\tau,x)|\leq \exp(E_{\alpha_2}\tau)$ (resp. $|\partial_{\alpha_2}\phi(\sigma,x)|\leq \exp(E_{\alpha_2}\sigma)$ ).
  \item $\partial_{\alpha_3}L(\phi(\tau,x))$. As $L$ is smooth in $D\times\Delta$, all these factors are bounded by a constant $M_{\alpha_3}$.
  \item $\exp\left( -\int_0^\tau L(\phi(\sigma,x))d\sigma\right)$. This factor is bounded by $\exp(M_0\tau)$. 
\end{enumerate}

Next, by remarking that a factor of the first type appears in each term of the expansion of $\partial_\alpha H$, and taking $N$ large enough, it is possible to conclude that the integral \cref{ro8} converges an is equal to $0$ for $x\in D\times \left\{ 0 \right\} $. Therefore, the partial derivative $\partial_\alpha K(x)$ exists, is continuous and is equal to $0$ on $D\times \left\{ 0 \right\}$.

\end{proof}

\section{Normal form and transition of a semi-hyperbolic vector field }\label{app:NFandT}
%\subfile{NFandT.tex}

In this section, we present a rather general framework for the computation of a $\Cl$ normal form and the corresponding transition of a vector fields with a semi-hyperbolic singularity. The contents of this section are not only relevant for the object studied in this document, but for more general systems as well, c.f. \cite{JardonThesis}. To make our computations simpler, we prove a lemma that allows us to ``partition'' a smooth function. As a simple example of this partition, let $f(u,v)$ be a smooth function on $\R^2$. We show that $f$ can be written as $f(u,v)=f_1(uv,u)+f_2(uv,v)$, where $f_1$ and $f_2$ are smooth. This type of result becomes useful when computing the transition map that we present in \cref{sec:transitions}. 

\subsection{Normal form}\label{sec:nfs}
%\subfile{NormalForms.tex}
Here we provide a $\Cl$ normal form of a semi-hyperbolic vector field which frequently appears in the analysis of slow-fast systems. The goal of obtaining such a normal form is that the computation of the corresponding transition becomes simpler.

\begin{proposition}\label{prop:nf_semihyp}
Let $\alpha$, $\beta=(\beta_1,\ldots,\beta_m)$ and $\gamma$ be non-zero constants, and consider the vector field $X$ given by

\eq{X:\begin{cases}
u' &= \alpha w u (1+f)+ w g\\
v_j' &= \beta_j w v_j (1+f) \\
w' &= \gamma w^2(1+f)\\
z' &= -\Lambda+h,
\end{cases}
}

where $j=1,2,\ldots,m$; where the functions $f=f(u,v,w,z)$, $g=g(u,v,w,z)$ and $h=h(u,v,w,z)$ are smooth functions which are flat at the origin of $\R^{m+3}$, and where $\Lambda=\Lambda(u,v,w,z)$ is a smooth function such that $\Lambda(0)=0$ and $\tparcs{\Lambda}{z}(0)>0$. Then there exist a $\Cl$ coordinates $(U,V_1,\ldots,V_m,W,Z)$ under which $X$ can be written as

\eq{
X_{\text{sh}}^N:\begin{cases}
U' &= \alpha W U \\
V_j' &= \beta_j W V_j \\
W' &= \gamma W^2\\
Z' &=-G Z,
\end{cases}
}

where $G=G(U,V,W)$ is a $C^\ell$ function such that $G(0)>0$.
\end{proposition}

\begin{proofof}{\cref{prop:nf_semihyp}}

From the definition of the vector field $X$ we note that the origin is a semi-hyperbolic singular point. The hyperbolic eigenspace is $1$-dimensional while the center eigenspace is $(m+2)$-dimensional. We now proceed in 4 steps as follows.

\begin{enumerate}[leftmargin=*]
\item Define a new vector field $Y$ by $Y=\frac{1}{1+f}X$, which reads as
\eq{
Y:\begin{cases}
u' &= \alpha w u + w \bar g\\
v_j' &= \beta_j w v_j  \\
w' &= \gamma w^2,\\
z' &= -\Lambda+\bar h,
\end{cases}
}

where the functions $\bar g$ and $\bar h$ are flat at the origin of $\R^{m+3}$. Note that in a small neighborhood of $(u,v,w,z)=(0,0,0,0)$ the vector fields $X$ and $Y$ are smoothly equivalent.

\item By looking at $DY(0)$, there exists an $(m+2)$-dimensional center manifold $\Wc$ \cite{Guckenheimer}. Let $M_0$ be the set of critical points of $Y$, that is

\eq{
M_0=\left\{ (u,v,w,z)\, | \, \Lambda(u,v,0,z)=0  \right\}.
}

By definition, the manifold $M_0$ is invariant and normally hyperbolic. Now, assume $|w|\ll 1$. This condition appears naturally in our applications. By Fenichel's theory \cite{Fenichel} the manifold $M_0$ persists as an invariant normally hyperbolic manifold $M_w$, for sufficiently small $w\neq 0$. We identify $M_w$ with $\Wc$. In other words, there exists a $C^\ell$ function $m=m(u,v,w)$ such that the center manifold $\Wc$ is given as a graph

\eq{
\Wc=\Graph(u,v,w,m).
}

Define $\zeta=z-m$, then $\zeta'=z'-m'$. But we know, due to invariance of $\Wc$ under the flow of $Y$, that $\zeta'|_{\zeta=0}=0$. This is, there exists a $C^\ell$ function $H=H(u,v,w,\zeta)$ such that $\zeta'=-H\zeta$. With $H(0)=0$ and $\parcs{H}{\zeta}(0)>0$. 

In conclusion of this step, there exists a $C^\ell$ transformation $\psi:(u,v,w,z)\mapsto (u,v,w,\zeta)$ that transforms the vector field $Y$ into
\eq{\tilde Y:\begin{cases}
u' &= \alpha w u + w \bar g\\
v_j' &= \beta_j w v_j  \\
w' &= \gamma w^2,\\
\zeta' &= -H\zeta,
\end{cases}
}

where $H=H(u,v,w,\zeta)$ is a $C^\ell$ function such that $H(0)=0$ and where $\parcs{H}{z}(0)=\parcs{\Lambda}{z}(0)>0$.

\item Observe that thanks to the previous step, the center manifold $\Wc$ has the simple expression $\Wc=\left\{ \zeta=0 \right\}$. We now want to separate the variables on the center manifold (these are $(u,v,w)$) from those on the hyperbolic subspace ($z$). Additionally, we want to keep the simple format that $\tilde Y$ has in the center direction. This amounts to find a change of coordinates along $\zeta$ only. For this we use an extension of Takens's theorem on semi-hyperbolic vector fields \cite{Takens_partially} due to Bonckaert \cite{Bonckaert1,Bonckaert2}. With this, it is possible to show there exists a $C^\ell$ transformation, fixing the center coordinates, that conjugates $\tilde Y$ to the vector field

\eq{\bar Y:\begin{cases}\label{bY}
u' &= \alpha wu + w \tilde g\\
v_j' &= \beta_j w v_j  \\
w' &= \gamma w^2,\\
Z' &= -\bar H Z,
\end{cases}
}

where now the flat perturbation $\tilde g$ is independent of $Z$ and $\bar H=\bar H(u,v,w)$ is a $\Cl$ function with $\bar H(0,0,0)>0$.

\item In this last step we eliminate the flat perturbation from $\bar Y$, which appears only along $u$. Due to the previous step, the dynamics on the center manifold are independent of $Z$. The restriction of $\bar Y$ to $\Wc$ reads as

\eq{\bar Y|_{\Wc}:\begin{cases}
u' &= \alpha wu +w\tilde g \\
v_j' &= \beta_j wv_j  \\
w' &= \gamma w^2.\\
\end{cases}
}

Note that for $w\neq 0$, the vector field $\frac{1}{w}\bar Y|_{\Wc}$ is hyperbolic. Let $\mathcal Y=\frac{1}{w}\bar Y|_{\Wc}$, that is

\eq{
  \mathcal Y:\begin{cases}
    u' &= \alpha u +\tilde g\\
    v_j' &= \beta_j v_j \\
    w' &= \gamma w.\\
  \end{cases}
}

Now we have the a result that shows that there exists a change of coordinates, respecting the variables $(v,w)$ that kills the term $\tilde g$. Keeping the coordinate $w$ fixed is important because we want to prove an equivalence relation with $w\mathcal Y$ and not with $\mathcal Y$. The following proposition shall appear in a general context in \cite{DumRouMaelBook}.

\begin{proposition}[{\sc{\cite{DumRouMaelBook}}}]\label{propRob1}
  There exists a diffeomorphism $(u,v,w)\mapsto(u+H(u,v,w),v,w)$ with $H$ flat at $(u,v,w)=0$ which brings $\mathcal Y$ to 
\eq{
  \bar{\mathcal Y}:\begin{cases}
    u' &= \alpha u \\
    v_j' &= \beta_j v_j \\
    w' &= \gamma w.\\
  \end{cases}
}

\end{proposition}

\begin{proof} We shall use the path method to show that $\bar {\mathcal Y}$ is conjugate to $\mathcal Y$. Let $s$ be a parameter and let us define the $s$-parameter family of vector fields
\eq{\label{aauu1}
\mathcal Y^s=\mathcal Y+s\tilde g\parc{u}
}

We call $\mathcal{Y}^s$ the path between $\mathcal Y$ and $\mathcal Y+\tilde g\tparc{u}$. We now look for an $s$-parameter family of diffeomorphisms $\mathcal H^s$ with $\mathcal H^0=\Id$ such that for each $s$ we have the conjugacy 
\eq{
  \mathcal H^s_*\mathcal Y=\mathcal Y^s.
}

In such a case, the vector fields $\mathcal Y$ and $\mathcal Y+\tilde g\tparc{u}$ are conjugated by $\mathcal H^1$. By derivation of the family $\H^s$ along $s$, we obtain an $s$-parameter family of vector field $\zeta^s$ satisfying
\eq{\label{aauu3}
  \zeta^s(\H^s)=\parcs{H^s}{s}.
}

This implies that by derivation of  \cref{aauu1} with respect to $s$ we obtain

\eq{\label{aauu2}
[\mathcal Y^s,\zeta^s]=\parcs{\mathcal Y^s}{s}=\tilde g\parc{u}.
}

Therefore, if are able to find a solution $\zeta^s$ of \cref{aauu2}, the conjugacy $\H^s$ is obtained by integration of \cref{aauu3}. In our particular case, we are looking for a solution along the $u$-direction, that is of the form $\zeta^s=P_s\parc{u}$. It follows that

\eq{
[\mathcal Y^s,\zeta^s] &= \left[ (\alpha u +s\tilde g)+\beta v\parc v+\gamma w\parc w, P_s\parc{u} \right]\\
&=\left(\mathcal Y^s(P_s)-\left(\alpha+s\parcs{\tilde g}{u}\right)P_s\right)\parc{u}.
}

Therefore we have reduced our conjugacy problem to solving the differential equation

\eq{\label{aauu4}
\mathcal Y^s(P_s)-\left(\alpha+s\parcs{\tilde g}{u}\right)P_s=\tilde g,
}

where we recall that $\tilde g=\tilde g(u,v,w)$ is flat at $(u,v,w)=(0,0,0)$. We now want to use \cref{propRob} to show that \cref{aauu4} has a solution $P_s=P_s(u,v,w)$ which is flat at $(u,v,w)=(0,0,0)$. For this, let $G_s=\alpha+s\tparcs{\tilde g}{u}$. Now, we only need a small adaptation: in the setting and notation of \cref{propRob}  we may assume (under the suitable arrangement of coordinates) that $\mathcal Y^s$ (or $X$ in \cref{propRob}) is tangent to $\R^d\times\left\{ 0 \right\}$ and $\left\{ 0 \right\}\times \R^{n-d}$. Let $\M_s^\infty(a)$ and $\M_s^\infty(b)$ denote the space of germs of $s$-families of smooth functions that are flat at $\left\{ a=0 \right\}$ and at $\left\{ b=0 \right\}$ respectively. Using a blowing-up at $0\in\R^n$ it can be shown that $\M_s^\infty(a,b)=\M_s^\infty(a)+\M_s^\infty(b)$ (see the arguments in \cref{prop:partition}). From this formula, it follows that it is sufficient to solve \cref{aauu4} in the spaces $\M_s^\infty(a)$ and $\M_s^\infty(b)$ respectively. Naturally, these two cases are equivalent up to the change of $\mathcal Y^s$ by $-\mathcal Y^s$ and $G_s$ by $-G_s$ in \cref{roprop}. In either case, the vector field $\mathcal Y^s$ (or  $-\mathcal Y^s$) of \cref{aauu4} satisfies the hypothesis of \cref{propRob}. Then for $\tilde g$ in $\M_s^\infty(a)$ (resp. in $\M_s^\infty(b)$) and  applying \cref{propRob}, we can solve \cref{aauu4} with $P_s$ in $\M_s^\infty(a)$ (resp. in $\M_s^\infty(b)$). 
  
\end{proof}

Thus, from \cref{propRob1}, we have that $\mathcal{Y}\sim\bar{\mathcal{Y}}$ respecting $w$, which implies $w\mathcal{Y}\sim w\bar{\mathcal{Y}}$. Therefore, we conclude that \cref{bY} can be written as stated in the proposition.

\end{enumerate}

\end{proofof}

\subsection{Partition of a smooth function}\label{sec:partition}
%\subfile{PartitionSmoothV2.tex}
In this section we investigate the problem of partitioning a smooth function. The result presented below is important since it is used to simplify the computation of transition maps. To be more specific, let us give a brief example. Consider the three dimensional differential equation
\eq{
x' &= x\\
y' &= -y\\
z' &= g(x,y)z,\\
}
where $g$ is a smooth function. We want to take advantage from the fact that $xy$ is a first integral. We show below that the function $g$ can be partitioned as $g(x,y)=g_1(xy,x)+g_2(xy,y)$. This makes the integration of $z'$ simpler.

\begin{lemma}\label{prop:partition} Let $u\in \R$ and $v\in\R^m$. Let $f=f(u,v)$ be a smooth function such that $f(0,0)=0$. Then there exist smooth functions $f_0=f_0(uv,u)$ and $f_1(uv,v)$ such that the function $f$ can be written as

\eq{
f=f_0+f_1,
}

where $f_0(0,0)=0$ and $f_1(0,0)=0$.

\end{lemma}

\begin{proofof}{\cref{prop:partition}} We proceed in two steps. The first consists in proving the formal version of the statement. The second step is to extend the formal result to the smooth case.\\

\subsubsection*{Formal step} % (fold)
\label{ssub:formal_step}
Let $\hat f$ denote the formal expansion of the smooth function $f$. Let $p\in\N$ and $q\in\N^m$. We use the following notation:

\begin{itemize}
  \item By $q\geq 0$ we mean $q_i\geq 0$ for all $i\in[1,m]$.
  \item For a vector $v\in\R^m$ we write $v^q=v_1^{q_1}\cdots v_m^{q_m}$.
  \item The $L_1$ norm of $q$ is denote by $|q|$, and thus for $q>0$ we have $|q|=\sum_{j=1}^mq_j$.
  \item We denote by $\tilde q_i$ the vector 
\eq{\tilde q_i=(q_1,\ldots,q_{i-1},q_{i+1},\ldots,q_m)}

and therefore we have that $v^{\tilde q_i}$ reads as 

\eq{v^{\tilde q_i}=\frac{v^q}{v_i^{q_i}}=v_1^{q_1}\cdots v_{i-1}^{q_{i-1}}v_{q_{i+1}}^{q_{i+1}}\cdots v_m^{q_m}.}

Besides, we have that the $L_1$ norm of $\tilde q_i$ is given by $|\tilde q_i|=|q|-q_i=\sum_{j=1,j\neq i}^mq_j$. \\
\end{itemize}

The formal series expansion of $f$ reads as

\eq{
\hat f= \sum_{p\geq 0, q\geq 0} a_{pq}u^pv^q,
}

where $a_{00}=0$. With the notation introduced above, we can partition $\hat f$ as follows

\eq{
\hat f = \sum_{p\geq |q|} a'_{pq}(uv)^qu^{p-|q|} + \sum_{i=1}^{m} \sum_{q_i\geq p+|\tilde q_i| }a'_{pq} (uv_i)^p (v_iv)^{\tilde q_i}v_i^{q_i-p-|\tilde q_i|},
}

where 

\eq{
(uv)^q &=(uv_1)^{q_1}\cdots (uv_m)^{q_m}\\
(v_iv)^{\tilde q_i}&=\frac{(v_iv)^{q}}{v_i^{2q_i}},
}

and where $a'_{pq}\in\R$ are suitable chosen coefficients. Let $r\in\N^m,\, s\in\N$. Define the following formal polynomials

\eq{
\hat h(uv,u) &= \sum_{r,s\geq0} \alpha_{rs}(uv)^ru^s = \sum_{p\geq |q|} a'_{pq}(uv)^qu^{p-|q|},}

where $\alpha_{rs}\in\R$, and

\eq{
\hat g_i(uv_i,v) &=\sum_{r,s,t\geq 0} \beta_{irs}(uv_i)^s v^r \\
&= \sum_{q_i\geq p+|\tilde q_i| }a'_{pq} (v_iv)^{\tilde q_i}(uv_i)^p v_i^{q_i-p-|\tilde q_i|},
} 

where $\beta_{irs}\in\R$. The coefficients $\alpha_{rs}$ and $\beta_{irs}$ are conveniently chosen to make the definitions hold. Let $uv=(uv_1,\ldots,uv_m)$. Define $\hat g=\hat g(uv,v)$ by $\hat g(uv,v)=\sum_{i=1}^m \hat g_i(uv_i,v)$, then we can write $\hat f$ as

\eq{
\hat f(u,v) &= \hat h(uv,u) +\hat g(uv,v).
}

This shows that the proposition holds for formal series.
% subsubsection formal_step (end)

\subsubsection*{Smooth step} % (fold)
\label{ssub:smooth_step}
By Borel's lemma \cite{Brocker}, there exist smooth functions $h=h(uv,u)$ and $g=g(uv,v)$ (whose formal series expansions are $\hat h$ and $\hat g$ respectively) such that

\eq{
f=h+g+R,
}

where $R$ (reminder) is a flat function. We now show the following.

\begin{proposition}\label{prop:smooth_step}
Let $u\in\R$, $v\in\R^m$, and $R(u,v)$ be a smooth flat function at $(0,0)\in\R\times\R^m$. There exist flat functions $r_0=r_0(uv,u)$ and $r_1=r_1(uv,v)$ such that
\eq{
R=r_0+r_1.
}
\end{proposition}

\begin{remark} \Cref{prop:smooth_step} together with the formal step $\hat f=\hat h+\hat g$ imply our result. 
\end{remark}

\begin{proofof}{\cref{prop:smooth_step}}

For this proof we shall use the blow-up technique. Let $\Phi:S^m\times\R^+\to\R^{m+1}$ be a blow-up map. The map $\Phi$ maps $S^m\times\left\{ 0 \right\}$ to the origin in $\R^{m+1}$. Let $\tilde R$ be a function defined by $\tilde R = R\circ\Phi$. Since $R$ is flat at the origin, the function $\tilde R$ is flat along the sphere $S^m$. We assume that the function $R=R(u,v)$ is defined on a small neighborhood $\mathcal R$ of the origin in $\R\times\R^{m}$; this neighborhood is defined as

\eq{
\mathcal R=\left\{ |u|\leq A,\, |v_i|\leq B_i \right\},
}
for some $A,B_i$ positive scalars. Let $0<\delta<1$. The sphere $S^m$ can be partitioned into $m+1$ regions as follows:

\eq{
\mathcal{U} &= S^m\backslash\left\{ |\bar u|\leq \delta \right\}\\
\mathcal{V}_i &= S^m\backslash\left\{ |\bar v_i|\leq \delta \right\},
}

where $(\bar u,\bar v)=(\bar u,\bar v_1, \ldots, \bar v_m)\in S^m$. We can then take a partition of unity to split $\tilde R$ as

\eq{\label{sp1}
\tilde R(\bar u,\bar v) = \tilde R_0(\bar u,\bar v) + \sum_{i=1}^m\tilde R_i(\bar u,\bar v),
}

where $\Supp(\tilde R_0)\subset\mathcal{U}$ and $\Supp(\tilde R_i)\subset\mathcal{V}_i$ for $i\in[1,m]$. We define as $R_0$ and $R_i$ the corresponding functions on $\R^{m+1}$ flat at the origin given by the blow-up map $\Phi$, that is $\tilde R_j=R_j\circ\Phi$, for $j=0,1,\ldots,m$. Note that  $R\to\tilde R$ is an isomorphism between the space of functions on $(u,v)\in\R^{m+1}$ flat at the origin, and the space of functions on $((\bar u,\bar v),\rho)\in S^m\times \R^+$ flat at $S^m\times\left\{ 0\right\}$. Therefore, the splitting \cref{sp1} induces the splitting

\eq{
R(u,v)= R_0(u,v) + \sum_{i=1}^m R_i(u,v)
}

of functions on $\R^{m+1}$. We will now prove that there exist flat functions $r_0$ and $r_i$ such that
\eq{
R_0(u,v) &= r_0(uv,u)\\
R_i(u,v)   &=  r_i(uv_i,v).
}

Let us detail only the case of $R_0$. The other functions are obtained in a similar way. 

The function $\tilde R_0$ has support in $\mathcal U$. We can parametrize $\mathcal U$ by the directional blow-up map $\Phi_{u}$ which reads as
\eq{
(\bar u,\bar v_1,\ldots,\bar v_m)\mapsto(\bar u,u\bar v_1,\ldots,u\bar v_m) = (u,v_1,\ldots,v_m).
}

Now, suppose that there exists a flat function $\tilde P_0$ defined by

\eq{
\tilde R_0 (u,\bar v) = \tilde P_0(u,u^2\bar v).
}

This implies that there is a function $\tilde r_0=\tilde P_0\circ\Phi_{ u}^{-1}$ such that
\eq{
R_0(u,v)=\tilde r_0(u,uv),
}

which is precisely what we want to prove. So, now we only need to show that indeed a function $\tilde P_0$ as above exists. For this let us define coordinates $(U,V_1,\ldots,V_m)$ given by

\eq{
U=u, \, V_1=u^2\bar v_1, \, \ldots, \, V_m = u^2\bar v_m,
}

and let $\tilde P_0(u,V)$ be a function defined as
\eq{
\tilde P_0(u,V)=\tilde R_0\left(\frac{V}{u^2},u\right).
}

Note that $\tilde P_0$ is flat at $(u,V)=0$. This is seen as follows. Since $\tilde R_0$ is flat along $\left\{u=0\right\}$, it follows that $\tilde P_0(0,0)=\tilde R_0|_{u=0}=0$ and 
\eq{
\parcs{\tilde P_0}{u}(0) &=\parcs{\tilde R_0}{u}|_{u=0}=0\\
\parcs{\tilde P_0}{V_i}(0) &= \frac{1}{u^2}\parcs{\tilde R_0}{\bar v_i}|_{u=0}=0,
}
and so on for the higher order derivatives.

Finally, for convenience of notation we define $r_0(uv,u)=\tilde r_0(u,uv)$, thus we can write $R_0(u,v)=r_0(uv,u)$ Following similar arguments as above we find the functions $r_i=r_i(uv_i,v)$ such that $R_i(u,v)=r_i(uv_i,v)$ for $i\in[1,m]$. Then we define $r_1(uv,v)=\sum_{i=1}^m r_i(uv_i,v)$. It follows that
\eq{
R(u,v)=r_0(uv,u)+r_1(uv,v).
}

\end{proofof}
% subsubsection smooth_step (end)

With this last proposition we can now write the function $f$ as
\eq{
f &= h(uv,u)+g(uv,v)+R(u,v)\\
&= h(uv,u)+g(uv,v)+r_0(uv,u)+r_1(uv,v).
}

Finally, to show the lemma we define the smooth functions $f_1$, $f_2$ of the statement by
\eq{
f_1 &= h +r_0\\
f_2 &=g+r_1.
}
\end{proofof}

\subsection{Transition}\label{sec:transitions}
%\subfile{Transitions.tex}
\newcommand{\wo}{w_{\text{out}}}

In this section we investigate the transitions for the vector field $X_{\text{sh}}^N$ computed in \cref{sec:nfs}. Relabeling the coordinates we recall that $X_{\text{sh}}^N$ reads as
\eq{
X_{\text{sh}}^N:\begin{cases}
u' &= \alpha w u \\
v_j' &= \beta_j w v_j \\
w' &= \gamma w^2\\
Z' &=-g Z,
\end{cases}
}

where $j=1,2,\ldots,m$, and where $g=g(u,v,w)$ is a $C^\ell$ function such that $g(0)=\Lambda>0$. We assume that $w\in\R^+$. For our applications, we are interested in only two particular situations.

\begin{enumerate}
\item The saddle 1 case where $\alpha=-1$, $\beta_j>0$ for all $j\in[1,m]$, and $\gamma>0$.
\item The saddle 2 case where $\alpha=1$, $\beta_j<0$ for all $j\in[1,m]$, and $\gamma<0$.
\end{enumerate}

\subsection*{Saddle 1} % (fold)
\label{sub:case_1_}

In this case we investigate the transitions of a vector field of the form
\eq{\label{eq:semihyp1}
Y:\begin{cases}
u' &= - w u \\
v_j' &= \beta_j w v_j \\
w' &= \gamma w^2\\
Z' &=-g Z,
\end{cases}
}

where the coefficients $\beta_j$, $\gamma$ are positive. Observe that the flow in the direction of $u$ and $Z$ is a contraction while it expands in all the other directions. Roughly speaking, this implies that a transition can go out at any expanding direction $v_j$ of $w$.\\

We investigate two types of transitions that are used in our applications. For this, let us define the following sections

\renewcommand{\wo}{w_{out}}
\newcommand{\vjo}{v_{j,out}}

\eq{
\Sigma_{\en} &=\left\{ (u,v,w,Z) \, | \, u=u_i\right\}\\
\Sigma_{\ex}^w &=\left\{ (u,v,w,Z) \, | \, w=\wo\right\}\\
\Sigma_{\ex}^{\pm v_j}&=\left\{ (u,v,w,Z) \, | \, v_j=\vjo\right\}.
}

In this section we compute the transitions
\eq{
\Pi^{w}  :& \Sigma_{\en}\to\Sigma_{\ex}^w\\
&(v,w,Z)\mapsto (\tilde u, \tilde v_i, \tilde Z),
}

for all $i\in[i,m]$, and
\eq{
\Pi^{\pm v_j} :& \Sigma_i\to\Sigma_{\ex}^{\pm v_j}\\
&(v,w,Z)\mapsto(\tilde u, \tilde v_i,\tilde w,\tilde Z),
}

for all $i\in[1,m]$ with $i\neq j$. 

\begin{proposition}\label{prop:tr_semihyp} Consider the vector field $Y$ given by \cref{eq:semihyp1} and  let $\Sigma_{\en}$, $\Sigma_{\ex}^w $, $\Sigma_{\ex}^{\pm v_j}$ and $\Pi^{w}$, $\Pi^{\pm v_j}$ be as above. Then 
\begin{itemize}
\item The transition $\Pi^{w}$ is given by
\eq{
\tilde u&=u\left( \frac{w}{\wo}\right)^{1/\gamma},\quad\qquad
\tilde v_i =v_i\left( \frac{\wo}{w}\right)^{\beta_i/\gamma}\\
\tilde Z &=Z\exp\left[ -\frac{\Lambda}{\gamma w}\left( 1+ \tilde\alpha w\ln(w)+w\tilde G \right)  \right]
}

where $\tilde \alpha=\tilde \alpha(uv_i^{1/\beta_i},uw^{1/\gamma})$ and $\tilde G=\tilde G(uv_i^{1/\beta_i},uw^{1/\gamma},\mu_i)$ are $C^\ell$ functions with $\mu_i=v_i^{1/\beta_i} w^{-1/\gamma}$.

\item The transition $\Pi^{\pm v_j}$ is given by
\eq{
\tilde u&=\left( \frac{v_j}{\eta_j} \right)^{1/\beta},\quad\qquad
\tilde v_i = v_i \left( \frac{\eta_j}{v_j} \right)^{\beta_i/\beta_j} ,\quad\qquad
\tilde w= w \left( \frac{\eta_j}{v_j} \right)^{\gamma/\beta_j}\\
\tilde Z &= Z\exp\left[ -\frac{\Lambda}{\gamma w} \left( 1+\tilde\alpha' w\ln(v_j)+w\tilde G'\right)\right],
}

with $i\neq j$ and where
\eq{
\tilde\alpha' &=\tilde\alpha'(uv_i^{1/\beta_i},uw^{1/\gamma})\\
\tilde G' &=\tilde G'(uv_i^{1/\beta_i},uw^{1/\gamma},\mu_w,\mu_i)
}

are $C^\ell$ functions with $\mu_w=w^{1/\gamma}v_j^{1/\beta_j}$ and $\mu_i=v_i^{1/\beta_i}v_j^{1/\beta_j}$.
\end{itemize}

\end{proposition} 

\begin{proofof}{\cref{prop:tr_semihyp}}
We detail first the computations for the transition $\Pi^{w}$. The transition $\Pi^{\pm v_j}$ is computed in a similar way so we only highlight the key parts of the computation.

\subsubsection*{The transition $\Pi^{w}$} % (fold)
\label{ssub:the_transition_}

In this case, the time of integration is $T=\ln\left( \frac{\wo}{w} \right)^{1/\gamma}$, where $\wo=w(t)|_{\Sigma_{\ex}^w}$ and $w=w(t)|_{\Sigma_{\en}}$. This time of integration is obtained form the equation $w'=\gamma w$. We also make the assumption that $v_i\in O(w^{\beta_i/\gamma})$. This assumption appears our applications, but roughly speaking it ensures that $\tilde v_i$ is well defined when $w\to 0$.  From the form of $Y$ we evidently have
\eq{
u(T) &=\tilde u = u\left( \frac{w}{\wo}\right)^{1/\gamma}\\
v_i(T) &= \tilde v_i = v_i\left( \frac{\wo}{w}\right)^{\beta_i/\gamma}.
}

It only remains to compute the transition for the $Z$ coordinate. Let us rewrite $Y$ as follows
\eq{
u' &= -u\\
v_i &= \beta_i v_i\\
w &= \gamma w\\
Z' &= -\frac{\Lambda+ G(u,v,w)}{w}Z,
}

where $G$ is a $C^\ell$ function vanishing at the origin. Observe that we have the first integrals $u^{b_i}v_i$ and $u^\gamma w$. We shall take advantage of such a fact. We define new coordinates $(U,V,W)$ given by
\eq{
U=u, \; V_i^{\beta_i} &= v_i, \; W^{\gamma}=w.
}

In these new coordinates we have the system

\eq{
U' &= -U\\
V_i' &= V_i\\
W' &= W\\
Z' &= -\frac{\Lambda+ G(U,V^{\beta_i},W^{\gamma})}{W^{\gamma}}Z.
}

In the new coordinates, the time of integration is given as $T=\ln\left( \frac{W_o}{W} \right)$. To have an idea of the expression of $\tilde Z$, let us first study a simplified scenario.

\subsubsection*{The case $G=0$} % (fold)
\label{ssub:the_case_g_0}

Let us suppose $G=0$. Therefore we have $Z'=-\frac{\Lambda}{W^\gamma}z$, which has the solution
\eq{
Z(t)=Z(0)\exp\left( -\Lambda \int_0^t W(s)^{-\gamma}ds \right),
} 

where $W(s)=W(0)\exp(s)$. Substituting the time of integration $T$ we have
\eq{\label{g0}
Z(T)=\tilde Z &= Z\exp\left(  -\frac{\Lambda}{W^\gamma}\int_0^{\ln\left( \frac{W_o}{W} \right)} e^{-\gamma s}ds \right)\\
&=Z\exp\left(  -\frac{\Lambda}{\gamma W^\gamma}\left( 1-\left( \frac{W}{W_o} \right)^\gamma\right) \right).
}

Observe that $\tilde Z\to 0$ as $W\to 0$. Let us now study the general case. We expect that the general case $G\neq 0$ is a perturbation of \cref{g0}.

% subsubsection the_case_g_0 (end)
\subsubsection*{The case $G\neq 0$} % (fold)
\label{ssub:the_case_gneq_0_}

We now consider that $G\neq 0$, we have
\eq{
Z(T)=\tilde Z = Z\exp\left( I_0+I_1 \right),
}

where 
\eq{
I_0 &= -\Lambda\int_0^T\frac{1}{W(s)}ds\\
I_1 &= \int_0^T\frac{G(U(s),V(s)^{\beta_i},W(s)^{\gamma})}{W(s)^{\gamma}}ds.
}

The integral $I_0$ has already been computed above. Let us write $F(U,V,W)=\frac{G(U(s),V(s)^{\beta_i},W(s)^{\gamma})}{W(s)^{\gamma}}$. We can do this because $G(U,0,0)=0$ and $V^{\beta_i}\in O(W^{\gamma})$. Now we estimate the integral $I_1$. Using \cref{prop:partition}, we can write
\eq{
I_1=\int_0^T \left[ F_1(s)+F_2(s)\right]ds,
}

where
\eq{
F_1 &= F_1(UV_1, \, \ldots, \, UV_m,\, UW,\, U)\\
F_2 &= F_2(UV_1, \, \ldots, \, UV_m,\, UW,\, V_1,\, \ldots,\, V_m,\, W).
}

Observe that $UW$ and all the $UV_j$'s' are first integrals. Let $J_1=\int F_1$ and $J_2=\int F_2$. Then we have
\eq{
J_1 &=\int _0^T F_1(UV,UW,U(s))ds\\\\
&=\int_0^{\ln\left( \frac{W_o}{W} \right)} F_1(UV,UW,Ue^{-s})ds.
}

Let us make the change of variables $y=e^{-s}$, we obtain
\eq{
J_1=-\int_{1}^{\frac{W}{W_o}}F_1(UV,UW,Uy)\frac{dy}{y}.
}

We expand the function $F_1$ in power of $y$ that is
\eq{
F_1(UV,UW,Uy)=F_1(UV,UW,0)+O(y).
}

Then we have
\eq{
J_1 &= -\int_{1}^{\frac{W}{W_o}} \alpha_1\frac{dy}{y}+\tilde F_1,
}
where $\alpha_1=\alpha_1(UV,UW)$ and $\tilde F_1=\tilde F_1(UV,UW,Uy(T))$ is some (unknown) $C^\ell$ function. Finally we get
\eq{
J_1=\alpha_1\ln\left( \frac{W_0}{W} \right)+\tilde F_1\left(UV,UW,U\frac{W}{W_0}\right).
}

The function $\tilde F_1$ is $C^\ell$ but unknown, and $W_0$ is a fixed positive constant, then we can simplify the notation of $\tilde F_1$ as $\tilde F_1 = \tilde F_1(UV,UW)$.\\

Next we have
\eq{
J_2 &=\int_0^T F_2(UV,UW,V(s),W(s))ds\\
&=\int_0^{\ln\left( \frac{W_o}{W} \right)} F_2(UV,UW,V_1e^{\beta_1 s},\ldots,V_me^{\beta_m s},We^{\gamma s})ds.
}

Let us make the change of variables $y=e^s$. Then we obtain
\eq{
J_2=\int_1^{\frac{W_o}{W}}F_2(UV,UW,V_1y^{\beta_1},\ldots,V_my^{\beta_m},Wy^{\gamma})\frac{dy}{y}.
}

As above, we expand in powers of $y$, that is
\eq{
F_2=\alpha_2+O(y),
}

and then we have
\eq{
J_2 = \alpha_2\ln\left( \frac{W_0}{W}\right)+\tilde F_2,
}

where $\alpha_2=\alpha_2(UV,UW)$, $F_2=F_2(UV,UW,\mu_i)$ is a $C^\ell$ function with $\mu_i=V_iW^{-1}$ for all $i\in[1,m]$. Recall that since $v_i\in O(w^{\beta_i/\gamma})$ we also have that $V\in O(W)$, that is $\mu_i$ is well defined.\\

Now we can write the integral $I_1$ as
\eq{
I_1 &= J_1+J_2\\
&=\alpha_1\ln\left( \frac{W_0}{W} \right)+\tilde F_1 + \alpha_2\ln\left( \frac{W_0}{W}\right)+\tilde F_2\\
&=\alpha \ln\left( \frac{W_0}{W}\right)+\tilde F,
}

where $\alpha=\alpha(UV,UW)$ and $\tilde F=\tilde F(UV,UW,\mu_i)$ are $C^\ell$ functions. Finally we write $\tilde Z$ in the original coordinates as follows
\eq{
\tilde Z &= Z\exp(I_0+I_1)\\
&=Z\exp\left[ -\frac{\Lambda}{\gamma w}\left( 1- \frac{w}{\wo}  \right) + \frac{1}{\gamma}\alpha \ln\left( \frac{\wo}{w}\right)+\tilde F \right]\\
&=Z\exp\left[ -\frac{\Lambda}{\gamma w}\left( 1+ \tilde\alpha w\ln(w)+w\tilde G \right)  \right],
}

where $\tilde \alpha=\tilde \alpha(uv_i^{1/\beta_i},uw^{1/\gamma})$ and $\tilde G=\tilde G(uv_i^{1/\beta_i},uw^{1/\gamma},\mu_i)$ are $C^\ell$ functions with $\mu_i=v_i w^{-\beta_i/\gamma}$.

% subsubsection the_case_gneq_0_ (end)

\subsubsection*{The transition $\Pi^{\pm v_j}$} % (fold)
\label{ssub:the_trainsition_pi_pm_v_j_}
In this case the time of integration is given by $T=\ln\left( \frac{\eta_j}{v_j} \right)^{1/\beta_j}$. Such a time of integration is obtained from the equation $v_j'=\beta_jv_j$. The we have
\eq{
\tilde u &= u \left( \frac{v_j}{\eta_j} \right)^{1/\beta}\\
\tilde v_i &= v_i \left( \frac{\eta_j}{v_j} \right)^{\beta_i/\beta_j}\\
\tilde w &= w \left( \frac{\eta_j}{v_j} \right)^{\gamma/\beta_j}.
}

It then only rests to compute $\tilde Z$. Following similar arguments as for the transition $\Pi^w$ we get in this case
\eq{
\tilde Z=Z\exp\left[ -\frac{\Lambda}{\gamma w} \left( 1+\tilde\alpha' w\ln(v_j)+w\tilde G'\right)\right],
}

where now 
\eq{
\tilde\alpha' &=\tilde\alpha'(uv_i^{1/\beta_i},uw^{1/\gamma})\\
\tilde G' &=\tilde G'(uv_i^{1/\beta_i},uw^{1/\gamma},\mu_w,\mu_i)
}

are $C^\ell$ functions with $\mu_w=wv_j^{-\gamma/\beta_j}$ and $\mu_i=v_iv_j^{-\beta_i/\beta_j}$.
% subsubsection the_trainsition_pi_pm_v_j_ (end)
% subsubsection the_transition_ (end)
\end{proofof}
% subsection case_1_ (end)

\subsection*{Saddle 2} % (fold)
\label{sub:case2_}
\newcommand{\uo}{u_{\text{out}}}
In this case we investigate the transitions of a vector field of the form
\eq{\label{eq:semihyp2}
Y:\begin{cases}
u' &= w u \\
v_j' &= -\beta_j w v_j \\
w' &= -\gamma w^2\\
Z' &=-g Z,
\end{cases}
}

where the coefficients $\beta_j$, $\gamma$ are positive. We assume that $u\in\R^+$. Observe that now, in contrast with case 1, we only have one expanding direction, which is $u$. This makes the study of the transition easier. Due to the same reason, it is more convenient to study a transition
\eq{
\Pi^u:\Sigma_{\en}\to\Sigma_{\ex},
}

where to be general, we let $\Sigma_{\en}$ be any codimension 1 subset of $\R^{m+3}$ obtained by setting one of the coordinates $(v,w)$ to a constant and with $u<\uo$; and where
\eq{
\Sigma_{\ex}=\left\{ (\ u,\tilde v,\tilde w,\tilde Z)\, | \, \tilde u=\uo \right\}.
}

\begin{proposition}\label{prop:semihyp2} Consider the vector field $Y$ given by \cref{eq:semihyp2} and let $\Sigma_{\en}$, $\Sigma_{\ex}$ and $\Pi^{u}$ be as above. Then
\eq{
\tilde v_i &=v_i\left( \frac{u}{\uo}\right)^{\beta_i}\\
\tilde w &=w\left( \frac{u}{\uo}\right)^{\gamma}\\
\tilde Z &=Z\exp\left[ -\frac{\Lambda}{\gamma w} \left( \left( \frac{\uo}{u}\right)^\gamma -1 +\alpha w\ln(u)+w\tilde F\right) \right]
}

where $\alpha=\alpha(u^{\beta_i}v_i,u^{\gamma}w)$ and $\tilde F=\tilde F(u^{\beta_i}v_i,u^{\gamma}w,u)$ are $C^\ell$ functions.
\end{proposition}

\begin{proofof}{\cref{prop:semihyp2}}
We have that the time of integration is $T=\ln\left( \frac{\uo}{u}\right)$. It follows that
\eq{
\tilde v_i &= v_i\left( \frac{u}{\uo}\right)^{\beta_i}\\
\tilde w &= w\left( \frac{u}{\uo}\right)^{\gamma}.
}

It only remains to compute $\tilde Z$. Following similar arguments as in case 1 we have
\eq{
\tilde Z=Z\exp\left[ -\frac{\Lambda}{\gamma w} \left( \left( \frac{\uo}{u}\right)^\gamma -1 +\alpha w\ln(u)+w\tilde F\right) \right],
}
where $\alpha=\alpha(u^{\beta_i}v_i,u^{\gamma}w)$ and $\tilde F=\tilde F(u^{\beta_i}v_i,u^{\gamma}w,u)$ are $C^\ell$ functions.
\end{proofof}

\section*{Acknowledgments}

H.J.K is partially supported by a CONACyT PhD grant.

% \section*{References}
% \bibliographystyle{plain}
% \bibliography{phdbib}

\end{document}